\documentclass{article}
\usepackage[utf8]{inputenc}
\usepackage{lipsum}
\usepackage[finnish,swedish,english]{babel}
\usepackage{url}
\usepackage{graphicx}
\usepackage{amsmath}
\usepackage{amsthm}
\usepackage{amssymb}
\usepackage{mathrsfs}
\usepackage{setspace}
\usepackage{color}
\usepackage{rotating}

\newtheorem{theorem}{Theorem}[section]
\newtheorem{lemma}[theorem]{Lemma}
\newtheorem{proposition}[theorem]{Proposition}
\newtheorem{corollary}[theorem]{Corollary}
\newtheorem{conjecture}[theorem]{Conjecture}

\newtheorem{definition}[theorem]{Definition}

\newtheorem{example}[theorem]{Example}

\author{Erik Sj\"oland}
\title{Enumeration of three term arithmetic progressions in fixed density sets}

\begin{document}

\newpage

\maketitle

\begin{abstract}
Additive combinatorics is built around the famous theorem by Szemer\'edi which asserts existence of arithmetic progressions of any length among the integers. There exist several different proofs of the theorem based on very different techniques. Szemer\'edi's theorem is an existence statement, whereas the ultimate goal in combinatorics is always to make enumeration statements. In this article we develop new methods based on real algebraic geometry to obtain several quantitative statements on the number of arithmetic progressions in fixed density sets. We further discuss the possibility of a generalization of Szemer\'edi's theorem using methods from real algebraic geometry.
\end{abstract}

\tableofcontents

\newpage

\section{Introduction}
\label{sec:density}
Finding an optimal lower bound for the number of arithmetic progressions of length 3 in any subset of fixed cardinality in a set is extremely difficult, and finding this for the cyclic group would imply close to optimal bounds to Szemer\'edi's theorem. In this paper we develop new methods that theoretically will achieve the optimal lower bound by Putinar's Positivstellensatz, and provide a relaxation scheme based on the Lasserre hierarchy providing lower bounds to the number of arithmetic progressions. The methods are general and can be applied to any fixed density set. We apply the methods to find the first known lower bound for the number of arithmetic progressions in any subset $S$ of the cyclic group $\mathbb{Z}_p$ with a fixed density $\frac{|S|}{p}$ for any prime $p$ (Theorem \ref{thm:density}), and investigate how much this can be improved using a similar certificate (Theorem \ref{thm:density2} and Corollary \ref{cor:density}). We discuss a possible way to improve the bounds further, which in theory could ultimately generalize Szemer\'edi's Theorem, in Section \ref{sec:discussion}. We provide sharper lower bounds with algebraic certificate for all primes $p \leq 17$ (Theorem \ref{thm:smallprimes}), and in Section \ref{sec:numerical2} we discuss lower bounds found using a numerical certificate for all primes $p \leq 613$ and the correct results up to $p \leq 32$ found by searching through all different possibilities. We have provided tables of all these exact lower bounds in the appendix.


\section{Results}
The easiest is always to consider small examples. Let $W(k,[n],D/n)$ denote the optimal lower bound for the number of arithmetic progressions of length $k$ in any subset of $[n]$ of cardinality $D$. Let $W(k,\mathbb{Z}_n,D/n)$ denote the optimal lower bound for the number of arithmetic progressions of length $k$ in any subset of $\mathbb{Z}_n$ of cardinality $D$. Using the gray code for fixed density necklaces that is introduced in Section \ref{sec:graycode} we can find the exact value for $W(k,\mathbb{Z}_n,D/n)$ when $n \leq 32$ and $D \in \{0,1,\dots,n\}$ for any $k \in \{1,\dots,n\}$. We have done this for $k=3,4,5$, and the results can be found in Tables \ref{tab:bubble3_1}, \ref{tab:bubble3_2}, \ref{tab:bubble4_1}, \ref{tab:bubble4_2}, \ref{tab:bubble5_1} and \ref{tab:bubble5_2} in the appendix. These results cannot be extended to all $n$, which is why the rest of the theorems involve lower bounds of $W(k,\mathbb{Z}_n,D/n)$.

The next theorem is the first theorem quantifying how many arithmetic progressions there are in any subset $S$ of $\mathbb{Z}_p$ of cardinality $D=|S|$. The bound holds for any prime $p$.

\begin{theorem}
 \label{thm:density}
Let $p$ be a prime number. A lower bound for the minimum number of arithmetic progressions of length $3$ among all subsets of $\mathbb{Z}_p$ of cardinality $D$,
\[  W(3,\mathbb{Z}_p,D/p) = \min \{ \sum_{\{i,j,k\} \textrm{ A.P. in } \mathbb{Z}_p} x_ix_jx_k : x_i \in \{0,1\} , \sum_{i=0}^{p-1} x_i = D\}, \]
is  
\[
\lambda = \frac{D^3 - (\frac{p+3}{2}) D^2 + (\frac{p+3}{2}-1)D}{p-1}.
\]
A certificate for the lower bound is given by:
\[ 
\begin{array}{rl}
\displaystyle \sum_{\{i,j,k\} \textrm{ A.P. in } \mathbb{Z}_p} X_iX_jX_k - \lambda =& \displaystyle  \sum_{i=0}^{p-1} \sigma_{1,i} X_i  + \sum_{i=0}^{p-1} \sigma_{2,i} X_i + \sigma_{3}(D- \sum_{i=0}^{p-1}X_i^3 ) \\
& \displaystyle  + \sigma_{4}(\sum_{i \neq j}X_i^2X_j - D(D-1)),
\end{array}
\]
where
\[ 
\begin{array}{rl}
\sigma_{1,i} &= \displaystyle  \frac{1}{p-1} \sum_{0<j<k<(p-1)/2} (X_{j+i} - X_{j+k+i} - X_{n-j-k+i} + X_{n-j+i})^2 \\ \\
\sigma_{2,i} &= \displaystyle  \frac{1}{p-1}(DX_i - \sum_{j=0}^{p-1} X_j)^2 \\ \\
\sigma_{3} &= \displaystyle \frac{(D-1)^2}{p-1} \\ \\
\sigma_{4} &= \displaystyle \frac{4D-p+3}{2(p-1)}.
\end{array}
\]
\end{theorem}

One can hope to find a sharper bound than the one provided in Theorem \ref{thm:density} that holds for all primes, but how to find this is not trivial. Next we provide sharper bounds for some low primes. The reason we can find nice algebraic bounds for $p \leq 17$ has to do with that the trigonometric functions in Theorem \ref{thm:density2} are on a nice form. We could find results for slightly larger $p$, but it would require more work and the bounds would contain messy combinations of trigonometric functions.

\begin{theorem}
 \label{thm:smallprimes}
Let $p$ be prime and let $W(3,\mathbb{Z}_p,D/p)$ denote the minimum number of arithmetic progressions of length $3$ among all subsets of $\mathbb{Z}_p$ of cardinality $D$. There are algebraic certificates with polynomials up to degree 3 giving the following bounds
\[  W(3,\mathbb{Z}_5,D/5) \geq \frac{D^3 - 3D^2 +2D}{6}, \]
\[ W(3,\mathbb{Z}_7,D/7) \geq \frac{D^3 - 4D^2 +3D}{8}, \]
\[ W(3,\mathbb{Z}_{11},D/11) \geq \frac{\sqrt{5}D^3+(15-12\sqrt{5})D^2+(-15+11\sqrt{5})D}{30}, \]
\[ W(3,\mathbb{Z}_{13},D/13) \geq \frac{21-2\sqrt{3}}{286}D^3 +\frac{28\sqrt{3}-151}{286}D^2+\frac{5-\sqrt{3}}{11}D \]
and
\[ W(3,\mathbb{Z}_{17},D/17) \geq \frac{1}{24}D^3 -\frac{1}{4}D^2+\frac{5}{24}D. \]
\end{theorem}

In fact, finding bounds sharper than those in Theorem \ref{thm:density} can be done very efficiently for relatively small primes using a simple degree 3 relaxation. As shown in the following theorem we can reformulate a certain relaxation as a linear program:
\begin{theorem}
\label{thm:density2}
Let $r$ be a primitive root of the prime $p$. Let further
\[
V_{ij} = \Big| \Big\{ \{0,1,r^i\} : \{0,1,r^i\}=\{0,r^t,r^{j+t}\} \textrm{ for } t=0,\dots,p-2 \Big\} \Big|
\]
for all $i,j \in \{0,\dots,p-1 \}$,
\[ 
C_{ij} =cos(\frac{2 \pi (i-1)(j-1)}{p-1})
\]
for all $i,j \in \{0,\dots,p-1\}$
\[
u = [u_0,u_1,\dots,u_\frac{p-3}{2},u_\frac{p-1}{2},u_\frac{p-3}{2},\dots,u_1]^T,
\]
\[
u_+ = 1^Tu=u_0+2u_1+\dots+2u_{(p-3)/2}+u_{(p-1)/2}
\]
and
\[ 
v_i = \left\{\begin{array}{rl} 1 & \displaystyle \textrm{ if } r^i=2 \\ 0 & \displaystyle \textrm{ otherwise.} \end{array}\right.
\]
for $i \in \{0,\dots,p-1\}$.

The following optimization problems attain the same optimal value:
\begin{itemize}
\item[(a)]
\[
 \max \{ \lambda : \sum_{\{i,j,k\} \textrm{ A.P. in } \mathbb{Z}_p} X_iX_jX_k - \lambda = S \} 
 \] 
where
\[
\begin{array}{rl}
S=& \displaystyle  \sum_{i=0}^{p-1} \sum_{j=0}^{p-1} (\sum_{k=1}^{p-1}a_{ijk}X_{k})^2 x_i  +b \sum_{i=0}^{p-1} (DX_i - \sum_{j=0}^{p-1} X_j)^2 X_i \\
& \displaystyle + c( \sum_{i=0}^{p-1}X_i^3 -D) + d(\sum_{i,j}X_i^2X_j - D(D-1))
\end{array}
\]
for $a_{ijk},c,d \in \mathbb{R}$ and $b \geq 0$.
\item[(b)]
\[
\max \{ \frac{u_+}{p-1}(D - 1 -\frac{u_0}{u_+}(p-1)   )D(D-1) : Cu \geq 0, Vu=v \}.
\]
\end{itemize}
\end{theorem}

The following corollary, which follows from Theorem \ref{thm:density} and Theorem \ref{thm:density2} shows that there is room to improve the bounds further.
\begin{corollary}
\label{cor:density}
Let $p$ be prime and denote the optimal value to problem (a) in Theorem \ref{thm:density2} by $\lambda_p(D)$:
\[
\lambda_p(D) = \max \{ \lambda : \sum_{\{i,j,k\} \textrm{ A.P. in } \mathbb{Z}_p} X_iX_jX_k - \lambda = S \} 
 \] 
where
\[
\begin{array}{rl}
S=& \displaystyle  \sum_{i=0}^{p-1} \sum_{j=0}^{p-1} (\sum_{k=1}^{p-1}a_{ijk}X_{k})^2 X_i  +b \sum_{i=0}^{p-1} (DX_i - \sum_{j=0}^{p-1} X_j)^2 X_i \\
& \displaystyle + c( \sum_{i=0}^{p-1}X_i^3 -D) + d(\sum_{i,j}X_i^2X_j - D(D-1))
\end{array}
\]
for $a_{ijk},c,d \in \mathbb{R}$ and $b \geq 0$. For all $p$
\[
\left\lceil \frac{p+3}{4} \right\rceil \leq \min \{ D \in \mathbb{Z}_+ : \lambda_p(D) > 0 \}  \leq \frac{p+3}{2}.
\]
\end{corollary}

\section{Background}
\subsection{Existence of arithmetic progressions}
\label{sec:EAP}
This section provides a literature review on existence type theorems about arithmetic progressions. We start with some basic definitions and discuss many of the major contributions to the field. 

\begin{definition}
For $S \subseteq \mathbb{N}$, let us define the \emph{(upper) density} 
\[
\rho(S) = \displaystyle\lim_{n \rightarrow \infty}\frac{|S\cap\{1,2,\dots,n\}|}{n}.
\]
\end{definition}

As an example, the density of the even positive integers is $\rho(2\mathbb{N}) = 1/2$. Formally we define arithmetic progressions for the integers and for finite groups in the following way:

\begin{definition}
An \emph{arithmetic progression of length $k \in \mathbb{Z}^+$} in the positive integers is given by $a, a+b,a+2b,\dots, a+(k-1)b$, where $a,b \in \mathbb{Z}^+$.
\end{definition}

\begin{definition}
An arithmetic progression of length $k \in \mathbb{Z}^+$ in a finite group $G$ is a set of $k$ distinct elements $\{a, b \cdot a, b\cdot b \cdot a,\dots, b^{k-1} \cdot a \}$, where $a \in G$ and $b \in G \smallsetminus \{0\}$. 
\end{definition}
To clarify, for example $\{1,2,3\}$, $\{1,3,2\}$, $\{2,1,3\}$, $\{2,3,1\}$, $\{3,1,2\}$ and $\{3,2,1\}$ should be considered as the same arithmetic progression as it is the same set of $3$ distinct elements, hence when we write 
\[ \sum_{\{i,j,k\} \textrm{ is an A.P. in } G} \]
only one representative for every arithmetic progression is used.

It is interesting to study arithmetic progressions since it measures how structured a set is. A set is usually either very random in nature or has a structure similar to that of an arithmetic progression. In additive combinatorics one tries to answer questions about the length of the longest arithmetic progressions contained in very unstructured sets, how structured sets are contained in generalizations of arithmetic progressions and many similar questions. In the last century a lot of results in additive combinatorics have been proven using probabilistic methods, Fourier-analytic methods, sum set estimates, graph theory, ergodic theory and algebraic methods. To answer questions about existence of arithmetic progression in subsets of the integers is one of the most difficult questions in additive combinatorics, and the question that has gotten the most attention. We refer to the book by Tao and Vu \cite{TaoVu2006} for an accessible introduction to the subject.

Van der Waerden proved in 1927 \cite{Waerden1927} that for any given positive integers $r$ and $k$ there is a number $N$ such that if the integers $\{1,\dots,N\}$ are colored in $r$ colors, then there exists a monochromatic arithmetic progression of length $k$. This is one of the key results of Ramsey theory. The smallest such $N$ will be denoted $W_r(k)$ and is called the \emph{van der Waerden number}. Already for $r=2$ we have a huge gap between the lower bound  
\[
W_2(k) \geq p2^p
\]
 for primes by Berlekamp \cite{Berlekamp1968}  and the upper bound 
 \[
 W_2(k) \leq 2^{2^{2^{2^{2^{k+9}}}}}
 \]
  by Gowers \cite{Gowers2001}. Finding bounds for van der Waerden-like numbers is still an active area of research, recent contributions include \cite{Landman2005, Graham2006, Ahmed2009, Ahmed2013}. 

Results about arithmetic progressions in fixed-density sets implies results about monochromatic arithmetic progressions in colorings and are thus more general. The first non-trivial result about existence of arithmetic progressions in sets is due to Roth in 1953. The proof is based on the Hardy-Littlewood circle method, and because of this result Roth was awarded the fields medal in 1958.
\begin{theorem}[Roth's Theorem \cite{Roth1953}]
If $S \subseteq \mathbb{N}$ and $\rho(S)>0$, then it follows that we can find arithmetic progressions in $S$ of length $3$.  
\end{theorem}

In 1969 Szemer\'edi strengthened Roth result to existence of arithmetic progressions of length 4 \cite{Szemeredi1969} and in 1975 he generalized it to arithmetic progressions of arbitrary length \cite{Szemeredi1975}. In order to prove the result Szemer\'edi developed the famous regularity lemma, which states that every large enough graph can be partitioned into subgraphs of about the same size so that the edges between the different subgraphs behave fairly randomly. For these and other combinatorial results he was awarded the Abel prize in 2012.
\begin{theorem}[Szemer\'edi's Theorem \cite{Szemeredi1975}] 
If $S \subseteq \mathbb{N}$ and $\rho(S)>0$, then it follows that we can find arbitrarily long arithmetic progressions in $S$. 
\end{theorem}

\begin{definition}
We denote by $N(k,\delta)$ the smallest positive integer such that for any $M \geq N(k,\delta)$ it holds that any subset of $\{1,\dots,M\}$ of cardinality $\delta M$ contains an arithmetic progression of length $k$ for $0 < \delta \leq 1$. Similarly we define $\delta(N,k)$ to be the smallest density such that any subset of $\{1,\dots,M \}$ of cardinality $\delta(N,k) M$ contains an arithmetic progression of length $k$ for any $M \geq N$.
\end{definition}
Szemer\'edi's theorem implies the existence of an $N(k,\delta)$ for any $k$ and $\delta$ but fails to provide any good bounds for this number. Another proof of Szemer\'edi's theorem using ergodic methods was mainly due to Furstenberg in 1982 \cite{Furstenberg1982}. The ergodic methods makes the proof easier, and they have implied several other results in additive combinatorics including a multi-dimensional version and a polynomial version of Szemer\'edi's theorem, and the related density Hales-Jewett theorem. Tao gave a quantitative version based on ergodic theory in 2006 \cite{Tao2006}, but no effective bound for $N(k,\delta)$ has been found using ergodic methods.

Even though Roth's theorem is old, there has been more recent progress on improving the quantitative bounds $N(3,\delta)$. Most results are formulated bounding $\delta(N,3)$ instead of $N(3,\delta)$, and so we follow this notation. Roth's original proof showed that there is a constant $C$ such that $\delta(N,3) \leq \frac{C}{\log  \log N}$. This was improved by Heath-Brown and Szemer\'edi \cite{Heath1987, Szemeredi1990} showing that there exist constants $C_1$ and $C_2$ such that $\delta(N,3) \leq \frac{C_1}{(\log N)^{C_2}}$ and later by Bourgain \cite{Bourgain1999} showing that $\delta(N,3) \leq C \left( \frac{ \log \log N}{\log N} \right)^{1/2}$ for some constant $C$. More recently Bourgain \cite{Bourgain2008} improved the bound to $\delta(N,3) \leq  \frac{C( \log \log N)^2}{(\log N)^{2/3}}$, then Sanders \cite{Sanders2012} improved it to $\delta(N,3) \leq  \frac{C}{(\log N)^{3/4-o(1)}}$ and then Sanders \cite{Sanders2011} improved it further to $\delta(N,3) \leq  \frac{C}{(\log N)^{1+o(1)}}$ for some large enough constant $C$. 

Versions of Roth's theorem for other groups have also been of interest lately. The result was generalized to hold for arbitrary abelian groups of odd order by Meshulam in 1995 \cite{Meshulam1995}. A way to generalize Roth's theorem is through generalizing the (triangle) removal lemma, which is related to Szemer\'edi's regulaity lemma. Green generalized the removal lemma to abelian groups \cite{Green2005_2}, and Kr\'al, Serra and Vela generalized it to hold for all finite groups, with a version of Roth's theorem for all finite groups as a consequence \cite{Kral2009}:
\begin{theorem}
Let $G$ be a finite group of odd order $N$ and $S$ a subset of its elements. If the number of solutions of the equation $x \cdot z = y^2$ with $x,y,z \in S$ is $o(N^2)$, then the size of $S$ is $o(N)$. 
\end{theorem}
To get the result to also hold for non-abelian groups they had to avoid using Fourier analysis. Other results in additive combinatorics for non-abelian groups are the results by Bergelson and Hindman \cite{Bergelson1992}, Gowers \cite{Gowers2008}, Tao \cite{Tao2008, Tao2010, Tao2013}, Sanders \cite{Sanders2010} and Solymosi \cite{Solymosi2013}. The results have been established long after the commutative counterparts, and often with much weaker bounds.

Let $a \uparrow b$ denote $a^b$, $a \uparrow b \uparrow c = a \uparrow (b \uparrow c )= a^{(b^c)}$, and so on. In 1998 \cite{Gowers1998} Gowers showed that there exists an absolute constant $C$ such that $\delta(N,4) \leq  (\log \log N)^{-C}$, which was a major improvement on the previous best bound. The bound was improved by Green and Tao in 2009 \cite{Green2009} to $\delta(N,4) \leq  e^{-C\sqrt{\log \log N}}$. In 2001 Gowers provided the first effective upper bound for $N(k,\delta)$ for $k>3$ using the Fourier analytic machinery he had developed and been awarded the fields medal for in 1998.
\begin{theorem}[\cite{Gowers2001}, Theorem 18.2]
\label{thm:Gowers}
Let $k$ be a positive integer and let $0 < \delta \leq 1/2$. Then
\[ N(k,\delta) = 2 \uparrow 2 \uparrow \delta^{-1} \uparrow2 \uparrow 2 \uparrow k+9. \]
\end{theorem}

Before stating our next theorem, let us state what (Gower's version of) Szemer\'edi's theorem implies about $W(k, [n],\delta)$.
\begin{corollary}
\label{cor:limit}
It follows by Theorem \ref{thm:Gowers} that
\[
\lim_{n \rightarrow \infty} \min \{ \delta \in \mathbb{R}_+ : W(k,[n],\delta) > 0 \} = 0.
\]
Conversely, if there are lower bounds $\lambda(k,[n],\delta) \leq W(k,[n],\delta)$ with the additional property that
\[
\lim_{n \rightarrow \infty} \min \{ \delta \in \mathbb{R}_+ : \lambda(k,[n],\delta) > 0 \} = 0
\]
holds, then Szemer\'edi's theorem would follow.
\end{corollary}
\begin{proof}
We want to show that for any $\epsilon > 0$ and $k \in \mathbb{Z}_+$ there exists $N_{k,\epsilon} \in \mathbb{R}_+$ such that 
\[
\min \{ \delta \in \mathbb{R}_+ : W(k,[n],\delta) > 0 \} < \epsilon
\]
for all $n > N_{k,\epsilon}$. By Theorem \ref{thm:Gowers} we know that if we pick $N_{k,\epsilon}= 2 \uparrow 2 \uparrow \epsilon^{-1} \uparrow2 \uparrow 2 \uparrow k+9$, then any subset $S \subset \{1,\dots,n\}$ with cardinality $|S| \geq \epsilon n$, where $n>N$, contains an arithmetic progression of length $k$, hence the desired inequality holds.

On the other hand, suppose 
\[
\lim_{n \rightarrow \infty} \min \{ \delta \in \mathbb{R}_+ : \lambda(k,[n],\epsilon) > 0 \} = 0.
\]
Since $\lambda(k,[n],\epsilon) \leq W(k,[n],\epsilon)$ it follows that
\[
\lim_{n \rightarrow \infty} \min \{ \delta \in \mathbb{R}_+ : W(k,[n],\epsilon) > 0 \} = 0,
\]
or in other words that for any $\epsilon > 0$ there exists an $N_{k,\epsilon}$ such that for any $n>N_{k,\epsilon}$ it holds that $S \subset \{1,\dots,n\}$ with cardinality $|S| = n \epsilon$ has an arithmetic progression of length $k$. This is equivalent to that there are bounds for Szmer\'edi's theorem on the form $N(k,\epsilon) = N_{k,\epsilon}$.
\end{proof}

Another proof strategy for Szemer\'edi's theorem relies on generalizing Szemer\'edi's regularity lemma to hypergraphs, and was developed independently by Gowers \cite{Gowers2007} and Nagle, R\"odl, Schacht and Skokan \cite{Rodl2004, Rodl2006, Nagle2006, Rodl2007_1, Rodl2007_2}.

All the results above on Szemer\'edi's theorem require a positive upper density, but this condition might not be necessary. An open problem is how dense does a set have to be in order to contain arithmetic progressions. The following has been conjectured by Erd\H{o}s and Tur\'an.
\begin{conjecture}[Erd\H{o}s-Tur\'an conjecture]
Let $S$ be a subset of the positive integers. $S$ contains arithmetic progressions of any length if $\displaystyle\sum_{x\in S} \frac{1}{x}$ diverges.
\end{conjecture}

The primes do not have positive density, but results on arithmetic progressions among them would follow from the Erd\H{o}s-Tur\'an conjecture. Existence of arithmetic progressions of length 3 among any positive proportion of the prime numbers was proven by Green \cite{Green2005} and quantitative improvements were made by Helfgott and Roton \cite{Helfgott2011}. The result for arithmetic progressions of arbitrary length was found using Szemer\'edi's regularity lemma and properties about the distribution of primes by Green and Tao \cite{Green2007}.
\begin{theorem}[Green-Tao theorem]
There exist arbitrarily long arithmetic progressions among the prime numbers.
\end{theorem}

So far we have only discussed results on upper bounds for $\delta(N,k)$, one can also try to find lower bounds. The lower bounds are still very far from the upper bounds, and surprisingly not much improvements have been done on the lower bounds for over 50 years. Let $C$ be a large enough constant. The lower bound for $k=3$ due to Behrend \cite{Behrend1946} 
\[ \delta(N,3) \geq C \frac{1}{\log^{1/4} N} \cdot \frac{1}{2^{2\sqrt{2}\sqrt{\log_2 N}}} \]
was the best bound for over 60 years until Elkin \cite{Elkin2011} recently improved it to
\[
\delta(N,3) \geq C \log^{1/4} N \cdot \frac{1}{2^{2\sqrt{2}\sqrt{\log_2 N}}}.
\]
For larger $k$ the we have the following bound by Rankin \cite{Rankin1961}:
\[ \delta(N,k) \geq \exp (-C (\log N)^{1/(k-1)}). \]
Soon after Elkin improved the bound for $k=3$, O'Bryant \cite{OBryant2011} improved the bounds for $k=3$ and general $k$ further:
For every $\epsilon>0$ and $N$ large enough
\[
\delta(N,3) \geq \left( \frac{6 \cdot 2^{3/4}\sqrt{5}}{e\pi^{3/2}}- \epsilon \right) 2^{-\sqrt{8 \log_2 N} + \frac{1}{4} \log_2 \log_2 N}
\]
and
\[
\delta(N,k) \geq C_k 2^{ -n2^{(n-1)/2} \sqrt[n]{\log_2 N} + \frac{1}{2n} \log_2 \log_2 N},
\]
where $C_k>0$ is an unspecified constant and $n= \lceil \log k \rceil$.

Szemer\'edi's theorem roughly states that any subset $S \subset \{1,\dots,N\}$ of the integers, however random and unstructured it is, will contain very structured subsets if $N$ and $|S|$ are large enough. One can ask how long arithmetic progressions there are in more structured sets, such as for example $S+S = \{s_1+s_2 | s_1,s_2 \in S \}$ and $lS = \{s_1+\dots + s_l | s_1,\dots,s_l \in S \}$. In fact, the additional structure allows us to find much longer arithmetic progressions, for example we have the following result from Bourgain \cite{Bourgain1990} on sums of sets:
\begin{theorem}
Let $p \geq 1$ be a prime nuber and $A,B$ additive sets in $\mathbb{Z}_p$ with $|A|, |B| \geq \delta p $ for some $C\frac{( \log \log p)^3}{\log p} < \delta \leq 1$ where $C>1$ is a large enough absolute constant. Then $A + B$ contains an arithmetic progression of length at least $ \exp(c(\delta  \log p)^{1/3})$ for some constant $c>0$.
\end{theorem}
Iterated sumsets have even more structure, and results include those of  S\'ark\"ozy \cite{sarkozy1989}, Lev \cite{Lev1997, Lev1999}, Nathanson and Rusza \cite{Nathanson2000, Nathanson2002}, and Szemer\'edi and Vu \cite{Szemeredi2006}. We state a quite general and strong theorem from \cite{Szemeredi2006}:
\begin{theorem}
Let $d \geq 1$. There exists constants $C_d,D_d > 0$ such that for any $l \geq 1$ and $A \subset \{1,\dots N \}$ of cardinality $|A| \geq C_d \frac{N}{l^d}$ and $|A| \geq 2$, $lA$ contains a proper arithmetic progression of length $D_dl|A|^{1/d}$
\end{theorem}

The number of arithmetic progressions in a fixed density set is closely related to the number of monochromatic arithmetic progressions in coloring of a set. All results above can easily be translated in terms of colorings.
\begin{definition}
A \emph{$c$-coloring} of a set $S$ is a map $\chi : S \rightarrow \{1,\dots,c\}$ splitting $S$ into $c$ color classes.
\end{definition}
\begin{proposition}
If there exist arithmetic progressions in any set $S \subseteq \mathbb{Z}^+$ with positive upper density $\rho(S)>0$ then there exist monochromatic arithmetic progressions in any $c$-coloring of the positive integers.
\end{proposition}
\begin{proof}
Suppose there exist arithmetic progressions in any set $S \subseteq \mathbb{Z}^+$ with positive upper density. Let $\chi(\mathbb{Z}^+)$ be a $c$-coloring of the positive integers. Since $c$ is finite we have 
\[
\max_{i \in \{1,\dots,c\} }  \frac{|\{ j \in \{1,\dots,n\}: \chi(j)=i \} | }{n} \geq \frac{1}{c}>0, 
\]
for any $n$, and thus for some $i$ the color class $S_i=\{ j \in \{1,\dots,n\}: \chi(j)=i \}$ has positive upper density. By assumption there are arithmetic progressions in $S_i$, and thus of color $i$.
\end{proof}

\subsection{From existence to counting}
\label{sec:CAP}
The ultimate goal in combinatorial problems is always to count the desired constellation. The existence results on arithmetic progressions provide very strong results on how many arithmetic progressions there are in certain limits, but provides no information about the amount of arithmetic progressions far from the limits. To make the statements precise, let us define the functions of interest:

\begin{definition}

Let $R(k,[n],c)$ denote the  minimal number of monochromatic arithmetic progressions of length $k$ in any $c$-coloring of $[n]=\{1,\dots,n\}$.

Let $R(k,G,c)$ denote the minimal number of monochromatic arithmetic progressions of length $k$ in any $c$-coloring of a finite group $G$.

Let $W(k,[n],\delta)$ denote the minimal number of arithmetic progressions of length $k$ in any subset $S \subset\{1,\dots,n\}$ of cardinality $|S|=n\delta$.

Let $W(k,G,\delta)$ denote the minimal number of arithmetic progressions of length $k$ in any subset $S \subset G$ of cardinality $|S|=|G|\delta$.
\end{definition}

It would of course be desirable to find the exact values of these functions for any possible inputs $k$, $n$, $c$, $\delta$ and $G$, but it is not likely that they will ever be found because of how complicated the functions are. Finding the exact functions would give us sharp upper bounds for Szemer\'edi's theorem for all $k$, which is known to be extremely difficult already for low values of $k$. Since there is little hope in finding the exact functions, we aim to find as good lower bounds to the functions as possible. Good lower bounds provide a better understanding of arithmetic progressions, and if they are good enough they may additionally improve the current best existence bounds. We state all the best bounds known for various inputs that we have been able to find in the literature. Upper bounds are also provided, and are usually achieved by colorings.

{\bf R(3,[n],2)}:
Upper and lower bound from \cite{Parrilo2008_2}.
\[ \frac{1675}{32768}n^2(1+o(1)) \leq R(3,[n],2) \leq \frac{117}{2192}n^2(1+o(1)) \]

{\bf R(3,$\mathbb{Z}_n$,2)}:
Upper and lower bound from \cite{Sjoland_cyclic} (The special case $n \mod 24 \in \{1, 5, 7, 11, 13, 17, 19, 23\}$ found first in \cite{Cameron2007}).
\[ n^2/8-c_1n+c_2 \leq R(3,\mathbb{Z}_n,2) \leq n^2/8-c_1n+c_3\] 
where constants depends on $n \mod 24$:
\[
\begin{array}{c|c|c|c}
n \mod 24 & c_1 & c_2 & c_3 \\
\hline
1,5,7,11,13,17,19,23 & 1/2 & 3/8 & 3/8 \\
8,16 & 1 & 0 & 0 \\
2,10 & 1 & 3/2 & 3/2 \\
4,20 & 1 & 0 & 2 \\
14,22  & 1 & 3/2 & 3/2 \\
3,9,15,21 &  7/6 & 3/8 & 27/8 \\ 
0 &  5/3 & 0  & 0 \\ 
12 &  5/3 & 0 & 18 \\ 
6,18 &  5/3 &1/2 & 27/2 \\ 
\end{array}
\]

{\bf R(3,$D_{2n}$,2)}:
Upper and lower bound from \cite{Sjoland_cyclic}.
\[ n^2/4-2c_1n+2c_2 \leq R(3,D_{2n},2) \leq n^2/4-2c_1n+2c_3\] 
where constants depends on $n \mod 24$:
\[
\begin{array}{c|c|c|c}
n \mod 24 & c_1 & c_2 & c_3 \\
\hline
1,5,7,11,13,17,19,23 & 1/2 & 3/8 & 3/8 \\
8,16 & 1 & 0 & 0 \\
2,10 & 1 & 3/2 & 3/2 \\
4,20 & 1 & 0 & 2 \\
14,22  & 1 & 3/2 & 3/2 \\
3,9,15,21 &  7/6 & 3/8 & 27/8 \\ 
0 &  5/3 & 0  & 0 \\ 
12 &  5/3 & 0 & 18 \\ 
6,18 &  5/3 &1/2 & 27/2 \\ 
\end{array}
\]

{\bf R(3,G,2)}:
Let $G$ be any finite group.  Let $G_k$ denote the set of elements of $G$ of order $k$, $N=|G|$ and $N_k=|G_k|$. Denote the Euler phi function $\phi(k)=|\{ t \in \{1,\dots,k\}: t  \textrm{ and } k \textrm{ are coprime}\}|$. Let $K=\{k \in \{5,\dots,n\} : \phi(k) \geq \frac{3k}{4}\}$. Lower bound from \cite{Sjoland_group}:
\[
\displaystyle \sum_{k \in K}  \frac{N\cdot N_k}{8}(1- 3\frac{k-\phi(k)}{\phi(k)}) \leq R(3,G,2)
\]

{\bf R(4,$\mathbb{Z}_n$,2)}:
Wolf \cite{Wolf2010} contributed with several bound on $R(4,\mathbb{Z}_n,2)$. The bounds were improved by Lu and Peng \cite{Lu2012}. Note that the results in their paper differs by a factor 2, this is because they count $\{a,b \cdot a, b\cdot b \cdot a\}$ and $\{b \cdot b \cdot a, b \cdot a, a\}$ as distinct arithmetic progressions. If $p$ is prime we have \cite[Theorem 1.1]{Lu2012}:
\[
\frac{7}{192}p^2(1+o(1)) \leq R(4,\mathbb{Z}_p,2) \leq \frac{17}{300}p^2(1+o(1)).
\]
For general $n$ we have \cite[Theorem 1.2 and Theorem 1.3]{Lu2012}:
\[
c_1n^2(1+o(1)) \leq R(4,\mathbb{Z}_n,2) \leq c_2n^2(1+o(1)).
\]
where constants depends on $n \mod 4$:
\[
\begin{array}{c|c|c}
n \mod 4 & c_1 & c_2 \\
\hline
1,3 & 7/192 & 17/300\\
0 & 2/66 & 8543/1452000 \\
2 & 7/192 & 8543/1452000 \\
\end{array}
\]
Furthermore we have \cite[Theorem 1.5]{Lu2012}
\[
\underline{\lim}_{n \rightarrow \infty} R(4,\mathbb{Z}_n,2) \leq \frac{1}{24},
\]
and the following conjecture \cite[Conjecture 1.1]{Lu2012}:
\[
\inf_n\{R(4,\mathbb{Z}_n,2) \} = \frac{1}{24}.
\]

{\bf R(4,[n],2)}:
Upper bound from \cite[Equation (12)]{Lu2012}:
\[ R(4,[n],2) \leq \frac{1}{72}n^2(1+o(1)). \]

{\bf R(5,$\mathbb{Z}_n$,2)}:
For odd $n$ we have \cite[Theorem 1.4]{Lu2012}:
\[
R(5,\mathbb{Z}_n,2) \leq \frac{3629}{131424}n^2(1+o(1)).
\]
For even $n$ we have \cite[Theorem 1.4]{Lu2012}:
\[
R(5,\mathbb{Z}_n,2) \leq \frac{3647}{131424}n^2(1+o(1)).
\]
Furthermore we have \cite[Theorem 1.5]{Lu2012}
\[
\underline{\lim}_{n \rightarrow \infty} R(5,\mathbb{Z}_n,2) \leq \frac{1}{72}.
\]

{\bf R(5,[n],2)}:
Upper bound from \cite[Equation (13)]{Lu2012}:
\[ R(5,[n],2) \leq \frac{1}{304}n^2(1+o(1)). \]

%

{\bf W(3,$\mathbb{Z}_p$,$\delta$)}
Let $p$ be prime. A lower bound is given by Theorem \ref{thm:density}:
\[
\frac{(\delta p)^3-\frac{p+3}{2}(\delta p)^2 +\frac{p+1}{2} (\delta p)}{p-1} \leq W(3,\mathbb{Z}_p,\delta).
\]

Since $\{a,b,c\}$ is an arithmetic progression in $S \in \{1,\dots,n\}$ if and only if $a+c=2b$ for distinct $a,b,c \in S$ there are many related problems where one tries to count the number of solutions to various linear equations. One example is Schur triples, which are triples $\{a,b,c\}$ such that $a+b=c$, have been studied in \cite{Robertson1998, Schoen1999, Datskovsky2003}. Other examples include sets with no solutions to $x+y=3z$ studied in \cite{Matolcsi2013} and monochromatic solution to equations in groups \cite{Cameron2007}.

\section{Methods from real algebraic geometry}
Tools from real algebraic geometry are used to prove the main results of this paper. Even though the main result can be understood without understanding polynomial optimization, it played a vital part when finding the certificates in the main theorems. Only a few of the most important theorems will be presented. For a more extensive survey on the topic we refer to the excellent paper by Laurent \cite{Laurent2009}. For a longer review of how we applied and implemented methods from real algebraic geometry to the specific problem we refer to \cite{Sjoland_methods}.

\subsection{Polynomial optimization}
\label{sec:poly}
We use polynomial optimization to count monochromatic arithmetic progressions, and so in this section we briefly review the topic.

Let $f(x),g_1(x),\dots,g_m(x)$ be polynomials. A problem on the form
\[
\begin{array}{rll}
\rho_* = \inf  & \displaystyle  f(x) \\
\textnormal{subject to} & \displaystyle  g_1(x) \geq 0, \dots, g_m(x) \geq 0, \\
& \displaystyle  x \in \mathbb{R}^n,
\end{array}
\]
is a polynomial optimization problem. One way to solve a polynomial optimization problem is by studying the related problem:
\[
\begin{array}{rll}
\rho^* = \sup  & \lambda \\
\textnormal{subject to} & \displaystyle \lambda-f(x) \geq 0, g_1(x) \geq 0, \dots, g_m(x) \geq 0, \\
& \displaystyle  \lambda \in \mathbb{R}, x \in \mathbb{R}^n.
\end{array}
\]

\begin{definition}
Let $g_1,\dots,g_m \in \mathbb{R}[x_1,\dots,x_n]$ be polynomials. The \emph{quadratic module} generated by $g_1,\dots,g_m$ is defined by:
\[
\mathrm{QM}(g_1,\dots,g_m) = \{\sigma_0 + \sum_{i=1}^m \sigma_ig_i | \sigma_0,\dots,\sigma_m \textrm{ are sums of squares} \} . 
\]
\end{definition}

\begin{definition}
A quadratic module $\mathrm{QM}(g_1,\dots,g_m) $ is \emph{Archimedean} if there is an $N \in \mathbb{N}$ such that
\[ N - \sum_{i=1}^n x_i^2 \in
\mathrm{QM}(g_1,\dots,g_m).
\]
\end{definition}
\begin{theorem}[Putinar's Positivstellensatz]
\label{thm:Putinar}
Let the set $K = \{ x \in \mathbb{R}^n | g_1(x)\geq 0, \dots, g_m(x) \geq 0\}$ be compact. Assume that the associated quadratic module  $\mathrm{QM}(g_1,\dots,g_m) $ is Archimedean. If $f$ is strictly positive on $K$, then it is possible to find sums of squares $\sigma_1,\dots,\sigma_m$ such that $f= \sigma_0 + \sum_{i=1}^m \sigma_i g_i$.
\end{theorem}

In other words, if the quadratic module of the polynomial constraints in the polynomial optimization problem is Archimedean it is possible to solve the polynomial optimization problem using
\[
\begin{array}{rl}
\rho^* = \sup & \displaystyle  \lambda \\
\textnormal{subject to} & \displaystyle  \lambda-f(X) = \sigma_0 + \sum_{i=1}^m \sigma_ig_i, \\
& \displaystyle \lambda \in \mathbb{R}, \\
& \displaystyle  \sigma_0,\dots,\sigma_m \textrm{ are sums of squares},
\end{array}
\]
where $\lambda$ is the actual parameter, and $X$ is a formal indeterminate.

Let $\rho_j$ be the optimal value of the related problem in which the degrees of the sums of squares are bounded:
\[
\begin{array}{rl}
\rho_j = \sup & \displaystyle  \lambda \\
\textnormal{subject to} & \displaystyle  \lambda-f(X) = \sigma_0 + \sum_{i=1}^m \sigma_ig_i \\
& \displaystyle \sigma_0 \textrm{ is a sum of squares of degree at most } j, \\
& \displaystyle  \sigma_i \textrm{ is a sums of squares of degree at most } j-\deg(g_i).
\end{array}
\]
It is easy to see that $\rho^* \geq \rho_j$ for any $j$. Furthermore, the following important convergence result follows directly from Putinar's positivstellensatz (we present a simplified version of Theorem 3.4 page 805 \cite{Lasserre2001}, and Theorem 4.1 page 79 \cite{Lasserre2010}):
\begin{theorem}
Let $f,g_1,\dots,g_m \in \mathbb{R}[x]$, let $K= \{g_1 \geq 0,\dots,g_m \geq 0 \}$ be compact and assume that $\mathrm{QM}(g_1,\dots,g_m)$ is Archimedean. Assume that $\rho_*$ and $\rho^*$ defined as above are finite and let $\rho_j$ denote the optimal value of the $j$th restriction defined above. Then $\rho_j \rightarrow \rho_*=\rho^*$ as $j \rightarrow \infty$.
\end{theorem}
The family of relaxations $\rho_j$ is often referred to as the Lasserre hierarchy.

Let $v_d$ denote the vector of all monomials of degree less than or equal to $d$. $\sigma$ is a sum of squares of degree $2d$ if and only if it is possible to find a positive semidefinite matrix $Q$ such that $\sigma = v_d^T Q v_d$. It follows that $\rho_j$ can be obtained by 
\[ 
\begin{array}{rl}
\rho_j =  \displaystyle \sup_{\lambda,\{ Q_i \}}  & \lambda \\
\textnormal{subject to} &  \displaystyle  f(X) - \lambda = v_{ \lfloor \frac{j}{2} \rfloor }^TQ_0v_{ \lfloor \frac{j}{2} \rfloor} + \sum_{i=1}^m v_{ \lfloor \frac{j- \deg(g_i)}{2} \rfloor}^TQ_iv_{  \lfloor \frac{j - \deg(g_i)}{2} \rfloor} g_i(X), \\
& \displaystyle  Q_i \succeq 0 \textnormal{ for } i = 0,1,\dots,m.
\end{array}
\]
We rewrite the problem in order to get it on the generic form of a semidefinite program. If we let $h(X)= \sum_\alpha [h(X)]_\alpha X^\alpha$, we see that 
\begin{equation}
\label{eq:dual3}
\begin{array}{rl}
\rho_j = \displaystyle \sup_{\lambda, \{Q_i\}}  & \lambda \\
\textnormal{subject to} & [\displaystyle f(X) - v_{ \lfloor \frac{j}{2} \rfloor }^TQ_0v_{ \lfloor \frac{j}{2} \rfloor} + \sum_{i=1}^m v_{ \lfloor \frac{j- \deg(g_i)}{2} \rfloor}^TQ_iv_{  \lfloor \frac{j - \deg(g_i)}{2} \rfloor} g_i(X)]_0 = \lambda, \\
& \displaystyle [f(X) - v_{ \lfloor \frac{j}{2} \rfloor }^TQ_0v_{ \lfloor \frac{j}{2} \rfloor} + \sum_{i=1}^m v_{ \lfloor \frac{j- \deg(g_i)}{2} \rfloor}^TQ_iv_{  \lfloor \frac{j - \deg(g_i)}{2} \rfloor} g_i(X)]_\alpha = 0 \\
& \textrm{ for all } \alpha, \\
& \displaystyle  Q_i \succeq 0 \textnormal{ for } i = 0,1,\dots,m. \\
\end{array}
\end{equation}
Since  $[f(X) - v_{ \lfloor \frac{j}{2} \rfloor }^TQ_0v_{ \lfloor \frac{j}{2} \rfloor} + \sum_{i=1}^m v_{ \lfloor \frac{j- \deg(g_i)}{2} \rfloor}^TQ_iv_{  \lfloor \frac{j - \deg(g_i)}{2} \rfloor} g_i(X)]_\alpha$ is a linear polynomial in the entries of the matrices $Q_i$ this is indeed a semidefinite program. 

 Any numerical solution to the semidefinite program includes a certificate based on positive semidefinite matrices. The certificate can be translated into a sum of squares based certificate. The sum of squares based certificate attained additionally serves as a lower bound to the original polynomial optimization problem since $\rho_* = \rho^* \geq \rho_j$. Sum of squares based certificates are very easy to check by hand, which is one of the major advantages of this method.

\subsection{Exploiting symmetries in semidefinite programming}
\label{sec:Sym}
This subsection can be skipped for the reader who just wants to understand the final proofs as these methods were just necessary to find the numerical solutions that allowed us to find the final certificates. The main result we review in this section is Theorem \ref{thm:Schrijver} as it is one of the strongest results when working on semidefinite programs with symmetries. We discuss the limitations of this Theorem when applied to our specific problem in Section \ref{sec:methodsdensity}. Because of these limitations we would barely gain anything by using Theorem \ref{thm:Schrijver} instead of the easier Lemma \ref{lem:groupaverage}. Although the lemma is a very basic result in representation theory, the numerical calculations done in this article would not have been possible without this lemma. 

Let  $C$ and $A_1, \dots, A_m$ be real symmetric matrices and let $b_1,\dots,b_m$ be real numbers. In this section the objective is to reduce the order of the matrices in the semidefinite programming problem
\[ \max \{\mathrm{tr}(CX) ~ | ~X \textrm{ positive semidefinite},  \mathrm{tr}(A_i X)=b_i \textrm{ for } i = 1, \dots, m\} \]
when it is invariant under the actions of a group.

As in \cite{KPSchrijver2007} and \cite{Klerk2011}, we use a $\ast$--representation to reduce the dimension of the matrices. For  a survey on $\ast$--algebras we refer to the book by Takesaki \cite{Takesaki2002}. The method we use as well as other efficient methods for invariant semidefinite programs are discussed in \cite{Bachoc2012}. Other important recent contributions include \cite{Kanno2001, Gatermann2004, Vallentin2009, Maehara2010, Murota2010, Riener2013}.

\begin{definition}
A \emph{matrix $\ast$-algebra} is a collection of matrices that is closed under addition, scalar and matrix multiplication, and transposition.
\end{definition}

Let $G$ be a finite group, which acts on a finite set $Z$ and let $S_{|Z|}$ be the group of all permutations of $Z$. Let $h$ be a homomorphism $h : G \rightarrow S_{|Z|}$ that takes any element $g \in G$ to a permutation $h_g=h(g)$ of $Z$ for which  $h_{g g'}=h_{g}h_{g'}$ and $h_{g^{-1}}=h_{g}^{-1}$. For every permutation $h_{g}$ we element-wise define the corresponding permutation matrix $M_{g} \in \{0,1\}^{|Z| \times |Z|}$ by
\[
(M_{g})_{i,j}=
\left\{
\begin{array}{rl}
 1 & \textrm{ if } h_{g}(i )= j, \\ 
 0 & \textrm{otherwise} 
\end{array}
\right.
\]
for all $i,j \in Z$.
The span of these matrices is the following matrix $\ast$-algebra
\[
\mathcal{A} = \left\{ \sum_{g \in G} \lambda_g M_{g} ~ | ~ \lambda_g \in \mathbb{R} \right\}.
\]
The matrices $X$ that commute with all permutation matrices are the \emph{invariant matrices} of $G$. The collection of all such matrices, 
\[
\mathcal{A'} = \{X \in \mathbb{R}^{n \times n} | XM=MX \textrm{ for all } M \in \mathcal{A} \},
\]
is called the \emph{commutant} of $\mathcal{A}$, and it is again a $\ast$-algebra. Denote the dimension of the commutant by $d=\dim\mathcal{A'}$.

Let $J$ be the matrix of size $|Z| \times |Z|$ of ones. The commutant has a basis of $\{0,1\}$-matrices $E_1,\dots,E_d$ such that $\sum_{i=1}^d E_i = J$.

We form a new normalized basis by
\[
B_i = \frac{1}{\sqrt{tr(E_i^TE_i)}}E_i 
\]
for which $\mathrm{tr}(B_i^TB_j) = \delta_{i,j}$ where $\delta_{i,j}$ is the Kronecker delta.

The \emph{multiplication parameters} $\lambda_{i,j}^k$ are then defined by 
\[B_iB_j = \sum_{k=1}^d \lambda_{i,j}^kB_k \]
for $i,j,k = 1,\dots,d$.

We define $d\times d$-matrices $L_1,\dots,L_d$ by
\[
(L_k)_{i,j} = \lambda_{k,j}^i 
\]
for $i,j,k =1,\dots,d$, which span
\[ 
\mathcal{L}=\{\sum_{i=1}^d x_iL_i : x_1,\dots,x_d \in \mathbb{R} \}.
\]

\begin{theorem}[\cite{KPSchrijver2007}]
The linear function $\phi : \mathcal{A'} \rightarrow \mathbb{R}^{d \times d}$ defined by $\phi(B_i) = L_i$ for $i=1,\dots,d$ is a bijection. The linear function additionally satisfies $\phi(XY)=\phi(X)\phi(Y)$ and $\phi(X^T)=\phi(X)^T$ for all $X,Y \in \mathcal{A}'$.
\end{theorem}

\begin{corollary}[\cite{KPSchrijver2007}]
$\sum_{i=1}^d x_iB_i$ is positive semidefinite if and only if $\sum_{i=1}^d x_iL_i$ is positive semidefinite.
\label{cor:schrijver}
\end{corollary}

The next lemma shows that we can find $X \in \mathcal{A}'$. It is then possible to use Corollary \ref{cor:schrijver} to reduce the order of the semidefinite constraint.

\begin{lemma}
\label{lem:groupaverage}
There is a solution $X \in \mathcal{A}'$ to a $G$-invariant semidefinite program
\[ 
\max \{\mathrm{tr}(CX) ~ | ~X \textrm{ positive semidefinite},  \mathrm{tr}(A_i X)=b_i \textrm{ for } i = 1, \dots, m\}. 
\]
\end{lemma}

\begin{proof}

Let $C,A_1,\dots,A_m$ be $|Z| \times |Z|$ matrices commuting with $M_g$ for all $g \in G$. If $X$ is an optimal solution to the optimization problem then the group average, $X'=\frac{1}{|G|} \sum_{g \in G}M_g X M_g^T$, is also an optimal solution:
It is feasible since 
\[
\begin{array}{rl}
\displaystyle \mathrm{tr}(A_jX') &= \displaystyle  \mathrm{tr}(A_j\frac{1}{|G|} \sum_{g \in G}M_g X M_g^T) \\
&= \displaystyle  \mathrm{tr}(\frac{1}{|G|} \sum_{g \in G}M_g A_jX M_g^T) \\
&= \displaystyle  \mathrm{tr}(\frac{1}{|G|} \sum_{g \in G}A_jX ) \\
&= \displaystyle  \mathrm{tr}(A_jX),
\end{array}
\]
where we have used that the well-known fact that the trace is invariant under change of basis. By the same argument  $\mathrm{tr}(CX') =\mathrm{tr}(CX)$, which implies that $X'$ is optimal. It is easy to see that $X' \in \mathcal{A}'$.
\end{proof}

All in all we get the following theorem:

\begin{theorem}[\cite{KPSchrijver2007}]
\label{thm:Schrijver}
The semidefinite program
\[
\max \{\mathrm{tr}(CX) ~ | ~X \succeq 0,  \mathrm{tr}(A_i X)=b_i \textrm{ for } i = 1, \dots, m\} 
\]
has a solution $X = \sum_{i=1}^d x_iB_i$ that can be obtained by
\[
\max \{\mathrm{tr}(CX) ~ | ~ \sum_{i=1}^d x_iL_i \succeq 0, \mathrm{tr}(A_i \sum_{j=1}^d B_jx_j)=b_i \textrm{ for } i = 1, \dots, m\}.
\]
\end{theorem}

When $d$ is smaller than $|Z|$ the theorem can be used to reduce the size of the linear matrix inequality, which in turn improves the computational efficiency. This allows us to solve many problems that would otherwise not be possible.

\section{Gray code, bubble language and fixed-density necklaces}
\label{sec:graycode}
The original Gray code was used by Frank Gray in a patent filed in 1947 and granted in 1953 \cite{Gray1953}. For a background on Gray codes we refer to Knuth's Art of Computer Programming \cite{Knuth2005}. In this chapter we present some recent development on a Gray code for fixed-density necklaces by Ruskey, Sawada and Wiliams. These new methods are useful for counting arithmetic progressions, and the author's contribution based on these methods can be found in Section \ref{sec:numerical2}.

To represent all subsets $S \subset \mathbb{Z}_n$ of density $\delta = |S|/n$ we use binary strings; when for example $n=5$, $k=3$ and $\delta=4/5$ all different possible $S$ are: $11110$, $11101$, $11011$, $10111$, $01111$. All these strings are just rotations of one another. In our application it is enough to consider the lexicographically smallest $01111$. To exclude rotation and generate all sets of fixed density is what is in the literature called to generate all \emph{fixed-density necklaces}. There is software using gray code for listing all fixed-density necklaces in constant amortized time implemented in C \cite{Sawada2012b}, based on theory developed in \cite{Ruskey2012} and \cite{Sawada2012a}. We only cover the material useful for fixed-density necklaces, and refer to the mentioned articles for the most general forms of the theorems.

\begin{definition}
A \emph{gray code} is a generation of all combinatorial possibilities in which only a constant amount of change is required to go from any binary string to the following binary string. If there in addition only is a constant change to go from the last binary string to the first binary string, the gray code is \emph{cyclic}.
\end{definition}

To understand the concept we begin with the example of generating binary $n$-tuples using gray code.

Denote a binary $n$-tuple by $a_{n-1}a_{n-2}\dots a_0$ where $a_i \in \{0,1\}$, and let $a_\infty = \displaystyle \sum_{i=1}^n a_i +1 \mod 2$ denote the parity. We can generate all possible $n$-tuples by only changing one number $a_i$ at the time using the following algorithm:
\begin{itemize}
\item[1.] Initiate with $a_i = 0$ for $i = 0,\dots,n-1$ and $a_\infty = 1$.
\item[2.]  Visit $a_1a_2\dots a_n$.
\item[3.]  If $a_\infty = 1$, set $j=0$, otherwise let $j=\displaystyle \min (i :  a_{i-1}=1)$.
\item[4.]  Terminate if $j=n$, otherwise let $a_j = 1-a_j$ and thus also $a_\infty = 1-a_\infty$ and return to Step 2.
\end{itemize}
For a discussion about why this generates a cyclic gray code, and for other algorithms to generate $n$-tuples, we refer to the material by Knuth (subsection 7.2.1.1. pages 1--39 \cite{Knuth2005}), here we just check how the algorithm works for $n=4$: 
\begin{example}
We run through the given algorithm with input $n=4$ to generate all binary $4$-tuples. \\
$0000$, $a_\infty = 1, j=0 \Rightarrow \textnormal{switch }a_0,a_\infty $ \\
$0001$, $a_\infty = 0, j=1 \Rightarrow \textnormal{switch }a_1,a_\infty $ \\
$0011$, $a_\infty = 1, j=0 \Rightarrow \textnormal{switch }a_0,a_\infty $ \\
$0010$, $a_\infty = 0, j=2 \Rightarrow \textnormal{switch }a_2,a_\infty $ \\
$0110$, $a_\infty = 1, j=0 \Rightarrow \textnormal{switch }a_0,a_\infty $ \\
$0111$, $a_\infty = 0, j=1 \Rightarrow \textnormal{switch }a_1,a_\infty $ \\
$0101$, $a_\infty = 1, j=0 \Rightarrow \textnormal{switch }a_0,a_\infty $ \\
$0100$, $a_\infty = 0, j=3 \Rightarrow \textnormal{switch }a_3,a_\infty $ \\
$1100$, $a_\infty = 1, j=0 \Rightarrow \textnormal{switch }a_0,a_\infty $ \\
$1101$, $a_\infty = 0, j=1 \Rightarrow \textnormal{switch }a_1,a_\infty $ \\
$1111$, $a_\infty = 1, j=0 \Rightarrow \textnormal{switch }a_0,a_\infty $ \\
$1110$, $a_\infty = 0, j=2 \Rightarrow \textnormal{switch }a_2,a_\infty $ \\
$1010$, $a_\infty = 1, j=0 \Rightarrow \textnormal{switch }a_0,a_\infty $ \\
$1011$, $a_\infty = 0, j=1 \Rightarrow \textnormal{switch }a_1,a_\infty $ \\
$1001$, $a_\infty = 1, j=0 \Rightarrow \textnormal{switch }a_0,a_\infty $ \\
$1000$, $a_\infty = 0, j=4 \Rightarrow \textnormal{terminate}.$ \\
As claimed all $16$ of the binary $4$-tuples has been generated, and in every step there is only $1$ bit that changes. Further note that to go from $1000$ to $0000$ we also need to change $1$ bit, hence the code in this example is cyclic. 
\end{example}

\begin{definition}
A \emph{bubble language} is a set of binary strings with either of the following properties:
\begin{itemize}
\item[1.] The first $01$ of any string can be switched to $10$ to obtain another string in the set.
\item[2.] The first $10$ of any string can be switched to $01$ to obtain another string in the set.
\end{itemize}
\end{definition}
A wide variety of families of combinatorial objects are bubble languages, we only give a few examples and one proof here, and refer to \cite{Ruskey2012} for omitted proofs and further examples. We consider fixed density necklaces as they represent all possibilities of fixed-density sets in our application. The proof for the following proposition also works for the lexicographically smallest representatives of all necklaces and for Lyndon words (restricting to the aperiodic necklaces).
\begin{proposition}
The set of the lexicographically smallest representatives of fixed-density necklaces is a bubble language.
\end{proposition}
\begin{proof}
We prove that if we switch the first $10$ to a $01$ in any necklace we get another necklace.

Consider a necklace $\alpha = a_1\dots a_n$. If $\alpha \neq 0^d1^{n-d}$, then suppose the first $10$ appears at positions $j$ and $j+1$. The string $\alpha$ must start with $a_1\dots a_i = 0^s1^t$ for some $s,t > 0$, $s+t=j$, because otherwise there would be a lexicographically smaller rotation of $\alpha$. Note also that if we switch the first $10$ to $01$ we get a string $\beta$ which is lexicographically smaller than $\alpha$. We need to check that every rotation of $\beta$ is lexicographically larger than $\beta$ itself. Rotating and switching $a_j$ with $a_{j+1}$ commutes, so to keep the notation simple, let us rotate first. Let $r_i = a_i\dots a_na_1\dots a_{i-1}$ denote the $i$th rotation of $\alpha$, let $w_i$ denote the string we get after swapping $a_j$ with $a_{j+1}$ in $r_i$, and let us consider 4 cases:
\begin{itemize}
\item[1.]  $w_2,\dots,w_{j-1}$: The prefix of $r_i$ is of the form $0^u1^v$ with $u+v=j-i+1$, and is larger than the prefix of $\alpha$. Since $i\leq j-1$ we have $u+v\geq2$ and the swap of $a_j$ and $a_{j+1}$ thus only removes a $1$ in the end of the prefix in both $r_i$ and $\beta$, which does not affect the lexicographic order of the two. We conclude that $\beta$ is lexicographically smaller than $w_1,\dots,w_{j-1}$.
\item[2.] $w_j$: If $s>1$ then $\beta$ starts with $00$  and $\beta$ is thus lexicographically smaller than $w_j$, which starts with $01$. If $s=1$, then since $\alpha$ is lexicographically smaller than $r_j$ we must have $a_{j+1}=0$ and $a_{j+2}\dots a_{2j}=1^t$. Thus we get that $\beta$ starts with $01^{t-1}0$ whereas $w_j$ starts with $01^{t+1}$. It is clear that $\beta$ is lexicographically smaller than $w_j$.
\item[3.] $w_{j+1}$: We note that $w_{j+1}$ starts with a $1$ whereas $\beta$ start with a $0$. $\beta$ is clearly lexicographically smaller than $w_{j+1}$.
\item[4.] $w_{j+2},\dots,w_n$: After the rotation the swap does not occur in the prefix of $r_i$, whereas it makes the prefix of $\alpha$ smaller. Thus since $\alpha$ is lexicographically smaller than $r_i$ we can conclude that $\beta$ is lexicographically smaller than $w_{j+2},\dots,w_n$.
\end{itemize}
\end{proof}

\begin{definition}
In \emph{co-lexicographic (co-lex) order} strings are sorted by increasing value of their last symbol, i.e. lexicographic order read from right to left.
\end{definition}
The co-lex order does not provide a gray code, but it is more intuitive to define how we can recursively generate all binary strings in co-lex order. Once we have defined everything properly it is easy to introduce the cool-lex order, which is just a small modification of co-lex order and provides the gray code we are interested in.

Any binary string can be decomposed in the form $0^s1^tg$ where $g$ is the suffix consisting of all remaining zeroes and ones. We denote the empty suffix by $\epsilon$. We have the following recursive formula for generating all fixed density necklaces in a co-lex order \cite[Page 7]{Ruskey2012}:

\begin{proposition}
Let initially $s=d$, $t=n-d$ and $g= \epsilon$. All necklaces of length $n$ and density $d$ are generated recursively in co-lexicographic order by
\[
\mathscr{L}(s,t,g) = \left\{ \begin{array}{ll}
\displaystyle 0^s1^tg, \mathscr{L}(s-1,1,01^{t-1}g),\dots,\mathscr{L}(s-1,t-j,01^jg) &  \displaystyle  \textrm{ if } s>0, \\
\displaystyle 1^t g & \displaystyle \textrm{ if } s = 0. \end{array} \right.
\]
where $j$ is the minimum value such that $0^{s-1}1^{t-j}01^j g$ is a fixed density necklace.
\end{proposition}

We can visualize the recursion with a \emph{computation tree}, in which every node is labeled with a suffix $g$ from a certain instance of the recursion. We label the root by $\epsilon$, and the children of any node are labeled from left to right by $01^i$ with $i$ decreasing. When the suffix tree is build we add an additional leaf to each node in the tree of the form $0^s1^tg$. The computation tree for necklaces of length $8$ with density $4$ can be found in Figure \ref{fig:comptree}. We can write it simpler by replacing the suffix at every node with the string in the leaf next to it. This simpler tree is called the \emph{compact computation tree}, and the compact computation tree for necklaces of length $8$ with density $4$ can be found in Figure \ref{fig:comptree2}. Note that by construction we get the co-lex order by a pre-order traversal of the compact computation tree. 

\begin{figure}
\includegraphics[scale=1]{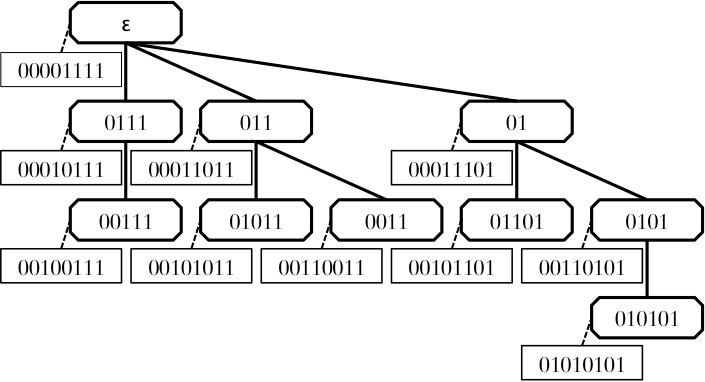}
\caption{The computation tree for necklaces of length 8 with density 4.}
\label{fig:comptree}
\end{figure}

\begin{figure}
\includegraphics[scale=1]{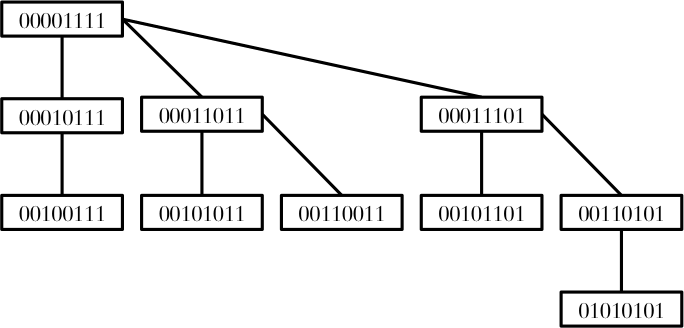}
\caption{The {\bf compact} computation tree for necklaces of length 8 with density 4.}
\label{fig:comptree2}
\end{figure}

Now we are ready to define the cool-lex order, which provides a gray code. Note that the only difference in that the term $0^s1^tg$ has been changed from first to last from the co-lex order.
\begin{definition}
Let initially $s=d$, $t=n-d$ and $g= \epsilon$. All necklaces of length $n$ and density $d$ are generated recursively in \emph{cool-lex order} by
\[
\mathscr{L}(s,t,g) = \left\{ \begin{array}{ll}
\mathscr{L}(s-1,1,01^{t-1}g),\dots,\mathscr{L}(s-1,t-j,01^jg, 0^s1^tg) & \textrm{ if } s>0, \\
1^t g & \textrm{ if } s = 0. \end{array} \right.
\]
where $j$ is the minimum value such that $0^{s-1}1^{t-j}01^j g$ is a fixed density necklace.
\end{definition}
Since the term $0^s1^tg$ is placed last in the recursion we get the cool-lex order by post-order transversal of the compact computation tree.

By either traversing the tree or using the recursive formula we can get the co-lex order and cool-lex order respectively for our example:
\[
\begin{array}{cc}
\textnormal{\bf co-lex} & \textnormal{\bf cool-lex} \\
00001111 & 00100111 \\
00010111 & 00010111 \\
00100111 & 00101011 \\
00011011 & 00110011 \\
00101011 & 00011011 \\
00110011 & 00101101 \\
00011101 & 01010101 \\
00101101 & 00110101 \\
00110101 & 00011101 \\
01010101 & 00001111 
\end{array}
\]

All we have done up to now could be done for any bubble language, not just for fixed density necklaces. By analyzing the possible steps from one binary string to the next in the compact computation tree in post-order one gets the following result (Theorem 3.1, page 10 \cite{Ruskey2012}):
\begin{theorem}
Cool-lex order provide a gray code for any bubble language.
\end{theorem}

To prove that a specific bubble language can be generated in constant amortized time requires more involved analysis, and this was done in the case for fixed-density necklaces (Theorem 2 page 11, \cite{Sawada2012a}):
\begin{theorem}
Fixed-density necklaces can be generated in cool-lex Gray code order or co-lex order in constant amortized time.
\end{theorem}

The methods in this section are used to obtain the results in Tables \ref{tab:bubble3_1}, \ref{tab:bubble3_2}, \ref{tab:bubble4_1}, \ref{tab:bubble4_2}, \ref{tab:bubble5_1} and \ref{tab:bubble5_2} in the appendix.

\section{Problem formulated as a semidefinite program}
\label{sec:methodsdensity}
Recall that  $W(k,G,\delta)$ denotes the minimal number of $k$-arithmetic progressions in the subset $S \subseteq G$ with $|S| = \delta |G|$. For this article we are primarily interested in $G = \mathbb{Z}_p$ for prime $p$. In the case when $k=3$ we have
\[ 
W(3,G,D/|G|) = \min \{\sum_{\{i,j,k\} \textrm{ is an A.P. in } G} x_ix_jx_k : x_i \in \{0,1\} , \sum_{i=0}^{|G|-1} x_i = D\},
\]
and using the methods introduced in Sections \ref{sec:poly} and \ref{sec:Sym} we make the following relaxation to get an algebraic certificate for the lower bound:
\[ 
W(3,G,D/|G|) \geq \max \{ \lambda : \sum_{\{i,j,k\} \textrm{ A.P. in } G} X_iX_jX_k - \lambda = S, X \in K \} 
\]
where $K = \{X \in [0,1]^n: \sum_{i=0}^{|G|-1}X_i = D\}$ and $S = \sum_i S_i^2g_i$ for $S_i,g_i \in \mathbb{R}[X_1,\dots,X_k]$ where $g_i$ are the half-spaces defining $K$.

Let us use the degree 3 relaxation of Putinar's Positivstellensatz, and let the maximal lower bound using this relaxation be denoted $\lambda^*$. Denote the elements of $G$ by $g_1,\dots,g_{|G|}$ and let $v=[1, X_{g_1}, \dots, X_{g_{|G|}}]^T$ be the vector of all monomials of degree less or equal to one. We get 
\[
\begin{array}{rl}
\lambda^* = \max  & \lambda \\
\textnormal{subject to:} & \displaystyle p_{_G} - \lambda = v^T Q_0 v + \sum_{g \in G} v^T Q_g^+ vX_g + \sum_{g \in G}  v^T Q_g^- v (1-X_g) \\
& \displaystyle + v^TQ_s^+v(\sum_{g \in G} X_g -D) + v^TQ_s^-v(D - \sum_{g \in G} X_g), \\
& \displaystyle X_{_G} \in [0,1]^{|G|}, \\
& \displaystyle Q_0,Q_s^+,Q_s^-,Q_g^+,Q_g^- \succeq 0 \textnormal{ for all } g\in G.
\end{array}
\]
This optimization problem is equivalent to a problem on the form of \eqref{eq:dual3}, that is if we use the notation  $h(x)= \sum_\alpha [h(x)]_\alpha x^\alpha$ then:
\[
\begin{array}{rl}
\lambda^* = \max  & \lambda \\
\textnormal{subject to} & [\displaystyle p_{_G} -  v^T Q_0 v + \sum_{g \in G} v^T Q_g^+ vX_g + \sum_{g \in G}  v^T Q_g^- v (1-X_g) \\
& \displaystyle + v^TQ_s^+v(\sum_{g \in G} x_g -D) + v^TQ_s^-v(D - \sum_{g \in G} X_g)]_0 = \lambda, \\
&  [\displaystyle p_{_G} -  v^T Q_0 v + \sum_{g \in G} v^T Q_g^+ vX_g + \sum_{g \in G}  v^T Q_g^- v (1-X_g) \\
& \displaystyle + v^TQ_s^+v(\sum_{g \in G} X_g -D) + v^TQ_s^-v(D - \sum_{g \in G} X_g)]_\alpha = 0, \\
& \displaystyle X_{_G} \in [0,1]^{|G|}, \\
& \displaystyle Q_0,Q_s^+,Q_s^-,Q_g^+,Q_g^- \succeq 0 \textnormal{ for all } g\in G.
\end{array}
\]
Since arithmetic progressions are invariant under affine transformations we can find a solution to the problem by restricting to invariant solutions. If $g_1$ is the identity element and $v_g = [1, X_{g_1 + g}, \dots, X_{g_{|G|}+g}]$  we get
\[
\begin{array}{rl}
\lambda^* = \max  & \lambda \\
\textnormal{subject to} & [\displaystyle p_{_G} - v^TQ_0v + \sum_{g \in G} v_g^TQ_{g_1}^+v_gX_g+ \sum_{g \in G} v^T_gQ_{g_1}^-v_g(1-X_g)  \\
& \displaystyle + v^TQ_s^+v(\sum_{g \in G} X_g -D) + v^TQ_s^-v(D - \sum_{g \in G} X_g)]_0 = \lambda, \\
&  [\displaystyle p_{_G} - v^TQ_0v + \sum_{g \in G} v_g^TQ_{g_1}^+v_g(1+X_g)+ \sum_{g \in G} v_g^TQ_{g_1}^-v_g(1-X_g)  \\
& \displaystyle + v^TQ_s^+v(\sum_{g \in G} X_g -D) + v^TQ_s^-v(D - \sum_{g \in G} X_g)]_\alpha = 0, \\
& \displaystyle X_{_G} \in [0,1]^{|G|} \\
& \displaystyle Q_0,Q_s^+,Q_s^-,Q_{g_{_1}}^+,Q_{g_{_1}}^- \succeq 0, \\
& \displaystyle Q_0(g_i,g_j) = Q_0(a + bg_i,a+bg_j), \forall \textnormal{ } (a,b) \in G \rtimes \mathbb{Z}^+, g_i,g_j \in G, \\
& \displaystyle Q_0(g,1) = Q_0(1,g) = Q_0(g_1,1), \forall \textnormal{ } g \in G, \\
& \displaystyle Q_s^+(g_i,g_j) = Q_s^+(a + bg_i,a+bg_j), \forall \textnormal{ } (a,b) \in G \rtimes \mathbb{Z}^+, g_i,g_j \in G, \\
& \displaystyle Q_s^+(g,1) = Q_s^+(1,g) = Q_s^+(g_1,1), \forall \textnormal{ } g \in G, \\
& \displaystyle Q_s^-(g_i,g_j) = Q_s^-(a + bg_i,a+bg_j), \forall \textnormal{ } (a,b) \in G \rtimes \mathbb{Z}^+, g_i,g_j \in G, \\
& \displaystyle Q_s^-(g,1) = Q_s^-(1,g) = Q_s^-(g_1,1), \forall \textnormal{ } g \in G, \\
& \displaystyle Q_{g_1}^+(g_j,g_k) = Q_{g_1}^+(b  g_j,b  g_k), \forall \textnormal{ } (a,b) \in G \rtimes \mathbb{Z}^+, g_j,g_k \in G,  \\
& \displaystyle Q_{g_1}^-(g_j,g_k) = Q_{g_1}^-(b  g_j,b  g_k), \forall \textnormal{ } (a,b) \in G \rtimes \mathbb{Z}^+, g_i,g_j \in G, \\
& \displaystyle Q_{g_1}^+(g,1)  = Q_{g_1}^+(b  g,1), \forall \textnormal{ } (a,b) \in G \rtimes \mathbb{Z}^+, g \in G, \\
& \displaystyle Q_{g_1}^-(g,1)  = Q_{g_1}^-(b  g,1), \forall \textnormal{ } (a,b) \in G \rtimes \mathbb{Z}^+, g \in G, \\
& \displaystyle Q_{g_1}^+(1,g)  = Q_{g_1}^+(1,b  g), \forall \textnormal{ } (a,b) \in G \rtimes \mathbb{Z}^+, g \in G,  \\
& \displaystyle Q_{g_1}^-(1,g)  = Q_{g_1}^-(1,b  g), \forall \textnormal{ } (a,b) \in G \rtimes \mathbb{Z}^+, g \in G. \\
\end{array}
\]
We see that this reduces the number of variables significantly and in this new formulation we only need to require that five $|G|+1 \times |G|+1$-matrices are positive semidefinite instead of the $2|G|+3$ matrices required in the original formulation. One could hope to reduce the size further using Theorem \ref{thm:Schrijver} if the commutant is small. Unfortunately though, one cannot hope to get reduce the size of the $Q_s^+$-matrices significantly since there will be one basis element of the commutant for the roughly $n/2$ pairs of $3$-tuples $(g_1,g_i,g_k)$ and $(g_1,g_i-g_k+g_i,g_i)$, and these matrices turns out to be the most essential to be able to find $\lambda^*$. We have solved this semidefinite program to find lower bounds to $W(3,\mathbb{Z}_p,D/p)$ for all primes $ 5 \leq p \leq 613$ and $D \in \{ 0, \frac{1}{20},\frac{2}{20}, \dots, p-\frac{1}{20}, p \}$ using the Matlab-based software CVX \cite{CVX} on the Triton computer cluster, which is part of the Aalto Science-IT project. The code was parallelized and we ran it simultaneously on 100 computer nodes. All the data was generated in 5 days. The data helped us get the intuition needed to prove the theorems in this section. The theoretically most interesting piece of this data can be found in Figure \ref{fig:allbounds} in Section \ref{sec:numerical2}. In that section the data is also discussed further.

To achieve better bounds we make a degree 5 relaxation of Putinar's Positivestellensatz. The SDP is the same as in the degree 3 relaxation above with the difference that $v$ is the vector of all monomials of degree up to $2$, and thus the $Q$-matrices are of size $|G|^2+|G|+1 \times |G|^2 + |G| +1$ and many more equalities. Just to get the equalities from the $Q$-matrices required 2000 lines of code, which was significantly more than in the degree 3 relaxation. It is not obvious to the author how the code could be simplified significantly as there are so many different cases to consider. Using the same computer cluster we found $W(k,\mathbb{Z}_p,D/p)$ for all primes $ 5 \leq p \leq 19$, $k \in \{3,4,5\}$ and $D \in \{ 0, \frac{1}{20},\frac{2}{20}, \dots, p-\frac{1}{20}, p \}$ by running parallelized code simultaneously on 100 computer nodes for 7 days. Again, the most interesting piece of data can be found in Figure \ref{fig:allbounds} in Section \ref{sec:numerical2}.

Since 
\[
\sum_{g \in G} x_g = D
\]
 on $\{0,1\}^{|G|}$ it follows that 
 \[
 \sum_{g \in G}  x_g^3 - D= 0,
 \]
 \[
 \sum_{g \neq h \in G}   x_g^2x_h - D(D-1) = 0
 \] 
 and 
 \[
 \sum_{g \neq h \neq l \in G} x_gx_hx_l - \binom{D}{3} = 0
 \]
  on $\{0,1\}^{|G|}$, providing several new possibilities for polynomials that are nonnegative on $\{0,1\}^{|G|}$. Note that these equation holds over the discrete hypercube, but not necessarily on the continuous hypercube, so technically they have to be added before making the relaxation. Similar assertion also holds for higher order polynomials. These redundant conditions simplifies the algebraic certificates we get from Putinar's Positivstellensatz, and replacing some of the original $g_i$s to these new constraints can in many cases simplify the computations significantly. By using these conditions and setting a lot of variables to zero we can make all polynomials homogeneous up to a constant and get lower bounds to $\lambda^*$. This is the main idea behind Theorem \ref{thm:density2}. As adding seemingly redundant conditions on this form gives more freedom in the semidefinite relaxation one gets from Putinar's Positivstellensatz we believe that degree 4 and 5 version of this theorem can be found using a similar trick. Numerical results suggests that a degree 4 and 5 certificate for a lower bound to the number of arithmetic progressions of length 3 would have many similar terms with certificates bounding arithmetic progressions of length 4 and 5 respectively.

\section{Proofs of Theorems \ref{thm:density}, \ref{thm:smallprimes}, \ref{thm:density2} and Corollary \ref{cor:density}}
\label{sec:proof2}
To simplify notation in the proofs of this section, let 
\[
\begin{array}{rl}
\displaystyle \omega(a^0,a^3,a^{21},a^{111}_1,\dots,a^{111}_d)  =& \displaystyle a^0+a^3\sum_{i=0}^{p-1}X_i^3+a^{21}\sum_{i \neq j}X_i^2X_j  +\\
 & \displaystyle +\sum_{i=1}^d a^{111}_i \sum_{(i_1,i_2,i_3) \in \mathrm{orb}_i} X_{i_1}X_{i_2}X_{i_3} 
 \end{array}
 \]
where $d$ is the number of orbits $\mathrm{orb}_i$ of the polynomials on the form $X_{i_1}X_{i_2}X_{i_3}$s under the action $\{i_1,i_2,i_3\} \mapsto \{ai_1+b,ai_2+b,ai_3+b\}$ with $(b,a) \in \mathbb{Z}_p \times (\mathbb{Z}_p \smallsetminus \{0 \})$.

\begin{proof}[Proof of Theorem \ref{thm:density}]
All we have to check is that the algebraic certificate
\[ 
\begin{array}{rl}
\displaystyle \sum_{\{i,j,k\} \textrm{ A.P. in } \mathbb{Z}_p} X_iX_jX_k - \lambda =& \displaystyle \sum_{i=0}^{p-1} \sigma_{1,i} X_i  + \sum_{i=0}^{p-1} \sigma_{2,i} X_i + \sigma_{3}(D- \sum_{i=0}^{p-1}X_i^3 ) \\
& \displaystyle + \sigma_{4}(\sum_{i \neq j}X_i^2X_j - D(D-1)) = S,
\end{array}
\]
where
\[ 
\begin{array}{rl}
\displaystyle \sigma_{1,i} &= \displaystyle  \frac{1}{p-1} \sum_{0<j<k<(p-1)/2} (X_{j+i} - X_{j+k+i} - X_{n-j-k+i} + X_{n-j+i})^2 \\
\displaystyle \sigma_{2,i} &= \displaystyle  \frac{1}{p-1}(DX_i - \sum_{j=0}^{p-1} X_j)^2 \\
\displaystyle \sigma_{3} &= \displaystyle \frac{(D-1)^2}{p-1} \\
\displaystyle \sigma_{4} &= \displaystyle \frac{4D-p+3}{2(p-1)}
\end{array}
\]
is correct. We have:
\[ \sum_{i=0}^{p-1}\sigma_{1,i}X_i = \omega(0,0,\frac{(\frac{p-3}{2})}{p-1},\frac{p-7}{p-1},-\frac{6}{p-1},\dots, -\frac{6}{p-1}),\]
\[ \sum_{i=0}^{p-1}\sigma_{2,i}X_i = \omega(0,\frac{(D-1)^2}{p-1},\frac{3-2D}{p-1},\frac{6}{p-1},\dots,\frac{6}{p-1}),\]
\[ \sigma_3(D-\sum_{i=0}^{p-1}X_i^3)= \omega(D\frac{(D-1)^2}{p-1},-\frac{(D-1)^2}{p-1},0,\dots,0)\]
and 
\[ \sigma_4(\sum_{i \neq j}X_i^2X_j - D(D-1))= \omega(-D(D-1)\frac{2D-(\frac{p+3}{2})}{p-1},0,\frac{2D-(\frac{p+3}{2})}{p-1},0,\dots,0). \]
Summing up we get 
\[
\begin{array}{rl}
S &= \displaystyle \omega(D\frac{(D-1)^2}{p-1}-D(D-1)\frac{2D-(\frac{p+3}{2})}{p-1},0,0,1,0,\dots,0) \\
&= \displaystyle \omega(- \frac{D^3 - (\frac{p+3}{2}) D^2 + (\frac{p+3}{2}-1)D}{p-1},0,0,1,0,\dots,0)  \\
&= \displaystyle \sum_{\{i,j,k\} \textrm{ A.P. in } \mathbb{Z}_p} X_iX_jX_k - \lambda.
\end{array}
\]
\end{proof}

\begin{proof}[Proof of Theorem \ref{thm:smallprimes}]
An algebraic certificate proving that 
\[
W(3,\mathbb{Z}_5,2) \geq \frac{D^3-3D^2+2D}{6}
\]
 is the following:
\[
\begin{array}{rl}
\displaystyle \sum_{\{i,j,k\} \textrm{ A.P. in } \mathbb{Z}_5} X_iX_jX_k - \lambda =& \displaystyle  \sum_{i=0}^4\sigma_{1,i}X_i +\sigma_2(D-\sum_{i=0}^4 X_i^3) \\
& \displaystyle + \sigma_3(\sum_{i \neq j} X_i^2X_j -D(D-1))  = S_5
\end{array}
\]
where
\[ 
\sigma_{1,i} = \frac{1}{6}(DX_i-\sum_{j=0}^4X_j)^2,
\]
\[
\sigma_{2} = \frac{D^2-2D+1}{6},
\]
\[
\sigma_{3} = \frac{2D-3}{6}
\]
and
\[
\lambda = \frac{D^3-3D^2+2D}{6}.
\]
To check that everything adds up, let us use the $\omega$-notation:
\[ 
\sum_{i=0}^4\sigma_{1,i}X_i= \omega(0,\frac{(D-1)^2}{6},-\frac{2D-3}{6},1)
\]
\[
\sigma_2(D-\sum_{i=0}^4X_i^3) =\omega(\frac{D(D-1)^2}{6},-\frac{(D-1)^2}{6},0,0)
\]
\[
 \sigma_3(\sum_{i \neq j} X_i^2X_j -D(D-1))  = \omega(-\frac{D(D-1)(2D-3)}{6},0,\frac{2D-3}{6},0).
\]
Summing up yields 
\[
\begin{array}{rl}
S_5 &= \displaystyle \omega(\frac{D(D-1)^2}{6}-\frac{D(D-1)(2D-3)}{6},0,0,1) \\
&= \displaystyle \omega(- \frac{D^3-3D^2+2D}{6},0,0,1)  \\
&= \displaystyle \sum_{\{i,j,k\} \textrm{ A.P. in } \mathbb{Z}_5} X_iX_jX_k - \lambda.
\end{array}
\]

An algebraic certificate proving that 
\[
W(3,\mathbb{Z}_7,2) \geq \frac{D^3-4D^2+3D}{8}
\]
 is the following:
\[
\begin{array}{rl}
\displaystyle \sum_{\{i,j,k\} \textrm{ A.P. in } \mathbb{Z}_7} X_iX_jX_k - \lambda =& \displaystyle  \sum_{i=0}^6\sigma_{1,i}X_i +\sigma_2(D-\sum_{i=0}^6 X_i^3)+\sum_{i=0}^6\sigma_{4,i}X_i  \\
& \displaystyle + \sigma_3(\sum_{i \neq j} X_i^2X_j -D(D-1))  = S_7
\end{array}
\]
where
\[ 
\sigma_{1,i} = \frac{1}{8}(DX_i-\sum_{j=0}^6X_j)^2,
\]
\[
\sigma_{2} = \frac{D^2-2D+1}{8},
\]
\[
\sigma_{3} = \frac{D-2}{4}
\]
\[
\sigma_{4,i} = \frac{1}{8}(X_{i+1}+X_{i+2}-X_{i+3}+X_{i+4}-X_{i+5}-X_{i+6})^2
\]
and
\[
\lambda = \frac{D^3-4D^2+3D}{8}.
\]
To check that everything adds up, let us use the $\omega$-notation:
\[ 
\sum_{i=0}^6\sigma_{1,i}X_i= \omega(0,\frac{(D-1)^2}{8},-\frac{2D-3}{8},\frac{3}{4},\frac{3}{4})
\]
\[
\sigma_2(D-\sum_{i=0}^6X_i^3) =\omega(\frac{D(D-1)^2}{8},-\frac{(D-1)^2}{8},0,0,0)
\]
\[
\sigma_3(\sum_{i \neq j} X_i^2X_j -D(D-1))  = \omega(-\frac{D(D-1)(D-2)}{4},0,\frac{D-2}{4},0,0)
\]
\[
\sum_{i=0}^6\sigma_{4,i}X_i  = \omega(0,0,\frac{1}{8},\frac{1}{4},-\frac{3}{4})
\]
Summing up yields
\[
\begin{array}{rl}
S_7 &= \displaystyle \omega(\frac{D(D-1)^2}{8}-\frac{D(D-1)(D-2)}{4},0,0,1,0) \\
&= \displaystyle \omega(- \frac{D^3-4D^2+3D}{8},0,0,1,0)  \\
&= \displaystyle \sum_{\{i,j,k\} \textrm{ A.P. in } \mathbb{Z}_7} X_iX_jX_k - \lambda.
\end{array}
\]

An algebraic certificate proving that 
\[
W(3,\mathbb{Z}_{11},2) \geq \frac{\sqrt{5}D^3+(15-12\sqrt{5})D^2+(-15+11\sqrt{5})D}{30}
\]
 is the following:
\[
\begin{array}{rl}
\displaystyle \sum_{\{i,j,k\} \textrm{ A.P. in } \mathbb{Z}_{11}} X_iX_jX_k - \lambda =& \displaystyle  \sum_{i=0}^{10}\sigma_{1,i}X_i +\sigma_2(D-\sum_{i=0}^{10} X_i^3) +\sum_{i=0}^{10}\sigma_{4,i}X_i \\
& \displaystyle + \sigma_3(\sum_{i \neq j} X_i^2X_j -D(D-1))  = S_{11}
\end{array}
\]
where
\[ 
\sigma_{1,i} = \frac{\sqrt{5}}{30}(DX_i-\sum_{j=0}^{10}X_j)^2,
\]
\[
\sigma_{2} = \frac{\sqrt{5}(D-1)^2}{30},
\]
\[
\sigma_{3} = \frac{2\sqrt{5}D+15-12\sqrt{5}}{30}
\]
\[
\sigma_{4,i} = \frac{9\sqrt{5}-15}{30}((\sum_{j=0}^9 \cos (4 \pi j/10)X_{2^{i+j}+1})^2 + (\sum_{j=0}^9 \sin (4 \pi j/10)X_{2^{i+j}+1})^2)
\]
and
\[
\lambda = \frac{\sqrt{5}D^3+(15-12\sqrt{5})D^2+(-15+11\sqrt{5})D}{30}.
\]
To check that everything adds up, let us use the $\omega$-notation:
\[ 
\sum_{i=0}^{10}\sigma_{1,i}X_i= \omega(0,\sqrt{5}\frac{(D-1)^2}{30},-\frac{\sqrt{5}(2D-3)}{30},\frac{\sqrt{5}}{5},\frac{\sqrt{5}}{5})
\]
\[
\sigma_2(D-\sum_{i=0}^{10}X_i^3) =\omega(\frac{\sqrt{5}D(D-1)^2}{30},-\frac{\sqrt{5}(D-1)^2}{30},0,0,0)
\]
\[
\begin{array}{rl}
\displaystyle \sigma_3(\sum_{i \neq j} X_i^2X_j -D(D-1))  = \omega( & \displaystyle  -\frac{D(D-1)(2\sqrt{5}D+15-12\sqrt{5})}{30},0, \\
& \displaystyle  \frac{2\sqrt{5}D+15-12\sqrt{5}}{30},0,0)
\end{array}
\]
\[
\sum_{i=0}^{10}\sigma_{4,i}X_i  = \omega(0,0,\frac{9\sqrt{5}-15}{30},1-\frac{\sqrt{5}}{5},-\frac{\sqrt{5}}{5})
\]
Summing up yields
\[
\begin{array}{rl}
S_{11} &= \displaystyle \omega(\frac{\sqrt{5}D(D-1)^2}{30}-\frac{D(D-1)(2\sqrt{5}D+15-12\sqrt{5})}{30},0,0,1,0) \\
&= \displaystyle \omega(- \frac{\sqrt{5}D^3+(15-12\sqrt{5})D^2+(-15+11\sqrt{5})D}{30},0,0,1,0)  \\
&= \displaystyle \sum_{\{i,j,k\} \textrm{ A.P. in } \mathbb{Z}_{11}} X_iX_jX_k - \lambda.
\end{array}
\]

An algebraic certificate proving that 
\[
W(3,\mathbb{Z}_{13},2) \geq  \frac{21-2\sqrt{3}}{286}D^3 +\frac{28\sqrt{3}-151}{286}D^2+\frac{5-\sqrt{3}}{11}D
\]
 is the following:
\[
\begin{array}{rl}
\displaystyle \sum_{\{i,j,k\} \textrm{ A.P. in } \mathbb{Z}_{13}} X_iX_jX_k - \lambda =& \displaystyle  \sum_{i=0}^{12}\sigma_{1,i}X_i +\sigma_2(D-\sum_{i=0}^{12} X_i^3) +\sum_{i=0}^{12}\sigma_{4,i}X_i\\
& \displaystyle + \sigma_3(\sum_{i \neq j} X_i^2X_j -D(D-1))   = S_{13}
\end{array}
\]
where
\[ 
\sigma_{1,i} = \frac{21-2\sqrt{3}}{286}(DX_i-\sum_{j=0}^{12}X_j)^2,
\]
\[
\sigma_{2} = \frac{(21-2\sqrt{3})(D-1)^2}{286},
\]
\[
\sigma_{3} = \frac{2(21-2\sqrt{3})D+28\sqrt{3}-151}{286}
\]
\[
\begin{array}{rl}
\displaystyle \sigma_{4,i} =& \displaystyle \frac{23-9\sqrt{3}}{286}((\sum_{j=0}^{11} \cos (4 \pi j/12)X_{2^{i+j}+1})^2 + X(\sum_{j=0}^{11} \sin (4 \pi j/12)X_{2^{i+j}+1})^2)\\
& \displaystyle + \frac{5-\sqrt{3}}{22}((\sum_{j=0}^{11} \cos (2 \pi j/12)X_{2^{i+j}+1})^2 + (\sum_{j=0}^{11} \sin (2 \pi j/12)X_{2^{i+j}+1})^2)
\end{array}
\]
and
\[
\lambda = \frac{21-2\sqrt{3}}{286}D^3 +\frac{28\sqrt{3}-151}{286}D^2+\frac{5-\sqrt{3}}{11}D.
\]
To check that everything adds up, let us use the $\omega$-notation:
\[ 
\begin{array}{rl}
 \displaystyle \sum_{i=0}^{12}\sigma_{1,i}X_i= \omega(& \displaystyle 0,(21-2\sqrt{3})\frac{(D-1)^2}{286},-\frac{(21-2\sqrt{3})(2D-3)}{286},\\
& \displaystyle 6\frac{21-2\sqrt{3}}{286},6\frac{21-2\sqrt{3}}{286},6\frac{21-2\sqrt{3}}{286})
\end{array}
\]
\[
\sigma_2(D-\sum_{i=0}^{12}X_i^3) =\omega(\frac{(21-2\sqrt{3})D(D-1)^2}{286},-\frac{(21-2\sqrt{3})(D-1)^2}{286},0,0,0,0)
\]
\[
\begin{array}{rl}
\displaystyle \sigma_3(\sum_{i \neq j} X_i^2X_j -D(D-1))  = \omega( & \displaystyle -\frac{D(D-1)((42-4\sqrt{3})D+28\sqrt{3}-151)}{286},0, \\
& \displaystyle  \frac{(42-4\sqrt{3})D+28\sqrt{3}-151}{286},0,0,0)
\end{array}
\]
\[
\sum_{i=0}^{12}\sigma_{4,i}X_i  = \omega(0,0,\frac{88-22\sqrt{3}}{286},\frac{80+6\sqrt{3}}{143},-6\frac{21-2\sqrt{3}}{286},-6\frac{21-2\sqrt{3}}{286})
\]
Summing up yields
\[
\begin{array}{rl}
S_{13} &= \displaystyle \omega(- \frac{21-2\sqrt{3}}{286}D^3 -\frac{28\sqrt{3}-151}{286}D^2-\frac{5-\sqrt{3}}{11}D,0,0,1,0,0)  \\
&= \displaystyle \sum_{\{i,j,k\} \textrm{ A.P. in } \mathbb{Z}_{13}} X_iX_jX_k - \lambda.
\end{array}
\]

An algebraic certificate proving that 
\[
W(3,\mathbb{Z}_{17},2) \geq  \frac{1}{24}D^3 -\frac{1}{4}D^2+\frac{5}{24}D
\]
 is the following:
\[
\begin{array}{rl}
\displaystyle \sum_{\{i,j,k\} \textrm{ A.P. in } \mathbb{Z}_{17}} X_iX_jX_k - \lambda =& \displaystyle  \sum_{i=0}^{16}\sigma_{1,i}X_i +\sigma_2(D-\sum_{i=0}^{16} X_i^3) +\sum_{i=0}^{16}\sigma_{4,i}X_i\\
& \displaystyle + \sigma_3(\sum_{i \neq j} X_i^2X_j -D(D-1))   = S_{17}
\end{array}
\]
where
\[ 
\sigma_{1,i} = \frac{1}{24}(DX_i-\sum_{j=0}^{16}X_j)^2,
\]
\[
\sigma_{2} = \frac{(D-1)^2}{24},
\]
\[
\sigma_{3} = \frac{2D-6}{24}
\]
\[
\sigma_{4,i} = \frac{1}{8}((\sum_{j=0}^{16} (-1)^jX_{3^{i+j}+1})^2
\]
and
\[
\lambda = \frac{1}{24}D^3 -\frac{1}{4}D^2+\frac{5}{24}D.
\]
To check that everything adds up, let us use the $\omega$-notation:
\[ 
 \sum_{i=0}^{16}\sigma_{1,i}X_i= \omega( 0,\frac{(D-1)^2}{24},-\frac{(2D-3)}{24},\frac{1}{4},\frac{1}{4},\frac{1}{4})
\]
\[
\sigma_2(D-\sum_{i=0}^{16}X_i^3) =\omega(\frac{D(D-1)^2}{24},-\frac{(D-1)^2}{24},0,0,0,0)
\]
\[
 \sigma_3(\sum_{i \neq j} X_i^2X_j -D(D-1))  = \omega( -\frac{D(D-1)(2D-6)}{24},0,\frac{2D-6}{24},0,0,0)
\]
\[
\sum_{i=0}^{16}\sigma_{4,i}X_i  = \omega(0,0,\frac{1}{8},\frac{3}{4},-\frac{1}{4},-\frac{1}{4})
\]
Summing up yields
\[
\begin{array}{rl}
S_{17} &= \displaystyle \omega(\frac{D(D-1)^2}{24} -\frac{D(D-1)(2D-6)}{24},0,0,1,0,0) \\
&= \displaystyle \omega(- \frac{1}{24}D^3 +\frac{1}{4}D^2-\frac{5}{24}D,0,0,1,0,0)  \\
&= \displaystyle \sum_{\{i,j,k\} \textrm{ A.P. in } \mathbb{Z}_{17}} X_iX_jX_k - \lambda.
\end{array}
\]

\end{proof}

To show Theorem \ref{thm:density2} let us first prove two lemmas:

\begin{lemma}
\label{lem:density1}
The following two problems are equivalent:
\begin{itemize}
\item[(a)]
\[
 \max \{ \lambda : \sum_{\{i,j,k\} \textrm{ A.P. in } \mathbb{Z}_p} X_iX_jX_k - \lambda = S \} 
 \] 
where
\[
\begin{array}{rl}
S =&  \displaystyle \sum_{i=0}^{p-1} \sum_{j=0}^{p-1} (\sum_{k=1}^{p-1}a_{ijk}X_{k})^2 X_i  + b \sum_{i=0}^{p-1} (DX_i - \sum_{j=0}^{p-1} X_j)^2 X_i \\
&+  c( \sum_{i=0}^{p-1}X_i^3 -D) + \displaystyle d(\sum_{i\neq j}X_i^2X_j - D(D-1))
\end{array}
\]
for $a_{ijk} \in \mathbb{R},c,d \in \mathbb{R}$ and $b \geq 0$.
\item[(b')]
Let $r$ be a primitive root of $p$.
\[
\max \{   b(D - 1 -u_0/b   )D(D-1) : U \succeq 0, Vu=v \} 
\] 
where $U$ is given by
\[ 
\left[ \begin{array}{ccccccc}
b(1-D)^2 & b(1-D) & b(1-D)& \cdots & b(1-D) & \cdots & b(1-D) \\
b(1-D) & u_0 & u_1 & \dots & u_\frac{p-1}{2} & \dots & u_1 \\
b(1-D) & u_1 & \ddots & \ddots & \ddots & \ddots & \vdots \\ 
\vdots & \vdots & \ddots & \ddots & \ddots & \ddots &  u_\frac{p-1}{2}\\ 
b(1-D) & u_{\frac{p-1}{2}} & \ddots & \ddots & \ddots & \ddots & \vdots  \\ 
\vdots & \vdots & \ddots & \ddots & \ddots & \ddots &  u_1 \\ 
b(1-D) & u_1 & \cdots & u_\frac{p-1}{2} & \cdots & u_1 & u_0  \\ 
\end{array} \right],
\]
$V$ is element-wise given by
\[
V_{ij} = \Big| \Big\{ \{0,1,r^i\} : \{0,1,r^i\}=\{0,r^t,r^{j+t}\} \textrm{ for } t=0,\dots,p-2 \Big\} \Big|
\]
for all $i,j \in \{0,\dots,p-1 \}$, $u$ is the vector
\[
 [u_0,u_1,\dots,u_\frac{p-3}{2},u_\frac{p-1}{2},u_\frac{p-3}{2},\dots,u_1]^T
\]
and $v$ is element-wise given by
\[ 
v_i = \left\{\begin{array}{rl} 1 & \displaystyle \textrm{ if } r^i=2 \\ 0 & \textrm{ otherwise.} \end{array}\right.
\]
for $i \in \{0,\dots,p-1\}$.
\end{itemize}
\end{lemma}
\begin{proof}
The polynomial $ \sum_{i=0}^{p-1} \sum_{j=0}^{p-1} (\sum_{k=1}^{p-1}a_{ijk}X_{k})^2 X_i $ in Problem (a) can be written in matrix form. There is a unique positive semidefinite matrix $U_i'$ such that if $\hat{X}_i = [X_0,\dots,X_{i-1},X_{i+1},\dots,X_{p-1}]$ then
\[ 
\hat{X}_i^T U_i' \hat{X}_i = \sum_{j=0}^{p-1} (\sum_{k=1}^{p-1}a_{ijk}X_{k})^2,
\] 
and thus
\[
\begin{array}{rl}
S =& \displaystyle \sum_{i=0}^{p-1}\hat{X}_i^TU_i'\hat{X}_iX_i +b \sum_{i=0}^{p-1} (DX_i - \sum_{j=0}^{p-1} X_j)^2 X_i \\
& \displaystyle + c(\sum_{i=0}^{p-1} X_i^3 - D) + d(\sum_{i \neq j} X_i^2X_j -D(D-1)).
\end{array}
\]

As arithmetic progressions are invariant under affine transformations; if $(a,b) \in \mathbb{Z} \rtimes \mathbb{Z}^+$, and $\{i,j,k\}$ is an arithmetic progression, then $(a,b) \cdot \{i,j,k\}=\{a+bi,a+bj,a+bk\}$ is also an arithmetic progressions; $\sum_{i=0}^{p-1}\hat{X}_i^TU_i'\hat{X}_iX_i$ can also be assumed to be invariant under affine transformations by Lemma \ref{lem:groupaverage}. We get all elements of an orbit by first considering the orbit of the multiplicative action on $\hat{X}_0^TU_0'\hat{X}_0X_0$, and then translate by the additive action to get the elements in $\hat{X}_a^TU_a'\hat{X}_aX_a$ for all $a \in \{0,\dots,p-1\}$. Since $\hat{X}_a^TU_a'\hat{X}_aX_a$ is just a permutation of $\hat{X}_0^TU_0'\hat{X}_0X_0$ it follows that $U_a' \succeq 0$ if and only if $U_0' \succeq 0$, hence requiring that $U_0'$ is positive semidefinite is sufficient. Let now $X = [X_0,X_{r^0},X_{r^1},\dots,X_{r^{p-2}}]^T$ where $r$ is a primitive root of $p$. Since $\hat{X}_0^TU_0'\hat{X}_0X_0$ is invariant under the multiplicative group it follows $U_0'(r^i,r^j)=U_0'(r^{i+k},r^{j+k})=U_0'(r^j,r^i)=U_0'(r^i,r^{2i-j})=U_0'(r^{i+k},r^{2i-j+k})$ for $k=1,\dots,p-2$ and hence
\[
\begin{array}{rl}
\displaystyle \hat{X}_0^TU_0'\hat{X}_0X_0 & \displaystyle = X^T\left[ \begin{array}{ccccccc}
0 & \cdots & \cdots & \cdots & \cdots & \cdots & 0 \\
\vdots & u_0' & u_1' & \dots & u_\frac{p-1}{2}' & \dots & u_1' \\
\vdots & u_1' & \ddots & \ddots & \ddots & \ddots & \vdots \\ 
\vdots & \vdots & \ddots & \ddots & \ddots & \ddots &  u_\frac{p-1}{2}' \\ 
\vdots & u_{\frac{p-1}{2}}' & \ddots & \ddots & \ddots & \ddots & \vdots  \\ 
\vdots & \vdots & \ddots & \ddots & \ddots & \ddots &  u_1' \\ 
0 & u_1' & \cdots & u_\frac{p-1}{2}' & \cdots & u_1' & u_0'  \\ 
\end{array} \right]XX_0 \\
&\displaystyle =X^TU'XX_0.
\end{array}
\]
We see that if we let $u_i= u_i'+b$ we get
\[
S = X^TUX  +   c( \sum_{i=0}^{p-1}X_i^3 -D) + \displaystyle d(\sum_{i\neq j}X_i^2X_j - D(D-1))
\]
where
\[ 
U =\left[ \begin{array}{ccccccc}
b(1-D)^2 & b(1-D) & b(1-D)& \cdots & b(1-D) & \cdots & b(1-D) \\
b(1-D) & u_0 & u_1 & \dots & u_\frac{p-1}{2} & \dots & u_1 \\
b(1-D) & u_1 & \ddots & \ddots & \ddots & \ddots & \vdots \\ 
\vdots & \vdots & \ddots & \ddots & \ddots & \ddots &  u_\frac{p-1}{2}\\ 
b(1-D) & u_{\frac{p-1}{2}} & \ddots & \ddots & \ddots & \ddots & \vdots  \\ 
\vdots & \vdots & \ddots & \ddots & \ddots & \ddots &  u_1 \\ 
b(1-D) & u_1 & \cdots & u_\frac{p-1}{2} & \cdots & u_1 & u_0  \\ 
\end{array} \right].
\]
Since
\[
\sum_{\{i,j,k\} \textrm{ A.P. in } \mathbb{Z}_p} X_iX_jX_k - \lambda = S
\]
we conclude that $b(1-D)^2+c=0$, $2b(1-D)+u_0+d=0$ and that $Vu=v$, where 
\[
V_{ij} = \Big| \Big\{ \{0,1,r^i\} : \{0,1,r^i\}=\{0,r^t,r^{j+t}\} \textrm{ for } t=0,\dots,p-2 \Big\}  \Big\} \Big|
\]
for all $i,j \in \{0,\dots,p-1 \}$,
\[
u = [u_0,u_1,\dots,u_\frac{p-3}{2},u_\frac{p-1}{2},u_\frac{p-3}{2},\dots,u_1]^T
\]
and
\[ 
v_i = \left\{\begin{array}{rl} 1 & \displaystyle \textrm{ if } r^i=2 \\ 0 & \textrm{ otherwise.} \end{array}\right.
\]
for $i \in \{0,\dots,p-1\}$.

Finally we see that
\[
-\lambda = -cD+-dD(D-1)= b(1-D)^2D+(2b(1-D)+u_0)D(D-1),
\]
and simplifying gives the desired objective function:
\[
\begin{array}{rl}
\lambda &= -b(1-D)^2D-(2b(1-D)+u_0)D(D-1)\\
&= (-b(D-1) +  2b(D-1) -u_0   )D(D-1) \\
&= b(D - 1 -u_0/b   )D(D-1).
\end{array}
\]

\end{proof}

\begin{lemma}
\label{lem:density2}
Let $r$ be a primitive root of $p$ and $b \in \mathbb{R}$. Let further
\[
V_{ij} = \Big| \Big\{ \{0,1,r^i\} : \{0,1,r^i\}=\{0,r^t,r^{j+t}\} \textrm{ for } t=0,\dots,p-2 \Big\} \Big|
\]
for all $i,j \in \{0,\dots,p-1 \}$,
\[
u = [u_0,u_1,\dots,u_\frac{p-3}{2},u_\frac{p-1}{2},u_\frac{p-3}{2},\dots,u_1]^T,
\]
\[
u_+ = 1^Tu=u_0+2u_1+\dots+2u_{(p-3)/2}+u_{(p-1)/2}
\]
and
\[ 
v_i = \left\{\begin{array}{rl} 1 & \displaystyle \textrm{ if } r^i=2 \\ 0 & \textrm{ otherwise.} \end{array}\right.
\]
for $i \in \{0,\dots,p-1\}$.
The following two problems are equivalent:
\begin{itemize}
\item[(b')]
\[
\max \{  b(D - 1 -u_0/b   )D(D-1) : U \succeq 0, Vu=v \} 
\] 
where $U$ is given by
\[ 
\left[ \begin{array}{ccccccc}
b(1-D)^2 & b(1-D) & b(1-D)& \cdots & b(1-D) & \cdots & b(1-D) \\
b(1-D) & u_0 & u_1 & \dots & u_\frac{p-1}{2} & \dots & u_1 \\
b(1-D) & u_1 & \ddots & \ddots & \ddots & \ddots & \vdots \\ 
\vdots & \vdots & \ddots & \ddots & \ddots & \ddots &  u_\frac{p-1}{2}\\ 
b(1-D) & u_{\frac{p-1}{2}} & \ddots & \ddots & \ddots & \ddots & \vdots  \\ 
\vdots & \vdots & \ddots & \ddots & \ddots & \ddots &  u_1 \\ 
b(1-D) & u_1 & \cdots & u_\frac{p-1}{2} & \cdots & u_1 & u_0  \\ 
\end{array} \right],
\]
\item[(b)]
\[
\max \{ \frac{u_+}{p-1}(D - 1 -\frac{u_0}{u_+}(p-1)   )D(D-1) : Cu \geq 0, Vu=v \} 
\]
where
\[ 
C_{ij} =cos(\frac{2 \pi (i-1)(j-1)}{p-1}).
\]
\end{itemize}
\end{lemma}

\begin{proof}
As in the previous lemma, let
\[
U'=U-\left[ \begin{array}{ccccccc}
b(1-D)^2 & b(1-D) & \cdots & b(1-D) \\
b(1-D) & b & \cdots & b \\
\vdots & \vdots & \ddots & \vdots \\ 
b(1-D) & b & \cdots & b 
\end{array} \right],
\]
or in other words
\[
U'= \left[ \begin{array}{ccccccc}
0 & \cdots & \cdots & \cdots & \cdots & \cdots & 0 \\
\vdots & u_0' & u_1' & \dots & u_\frac{p-1}{2}' & \dots & u_1' \\
\vdots & u_1' & \ddots & \ddots & \ddots & \ddots & \vdots \\ 
\vdots & \vdots & \ddots & \ddots & \ddots & \ddots &  u_\frac{p-1}{2}' \\ 
\vdots & u_{\frac{p-1}{2}}' & \ddots & \ddots & \ddots & \ddots & \vdots  \\ 
\vdots & \vdots & \ddots & \ddots & \ddots & \ddots &  u_1' \\ 
0 & u_1' & \cdots & u_\frac{p-1}{2}' & \cdots & u_1' & u_0'  \\ 
\end{array} \right]
\]
where $u_i' = u_i-b$. Maximizing the objective function is equivalent to maximizing $b$ when $D - 1 -u_0/b \geq 0 $ . Let $U''$ be the matrix $U'$ without its first row and column. $U''$ is a circulant matrix, i.e. row $i$ is the first row shifted $i-1$ steps to the right. It is easy to see that $U$ is positive semidefinite if and only $U''$ is positive semidefinite. $U''$ is positive semidefinite if and only if all its eigenvalues are nonnegative. Let $\omega_j = e^{\frac{2\pi i j}{p-1}}$, where $i=\sqrt{-1}$. The eigenvalues of a circulant matrix with the first row $[k_0,k_{p-2},k_{p-3},\dots,k_1]$ are known to be on the form $\lambda_j =k_0+k_{p-2}\omega_j + k_{p-3}\omega_j^2+ \dots + k_1\omega_j^{p-2}$ for $j=0,\dots,p-2$. In our case $u'_i=u'_{p-1-i}$ for $i=1,\dots,\frac{p-3}{2}$. By Euler's formula we get
\[
\begin{array}{rcl}
\displaystyle \lambda_j &=& \displaystyle u'_0+u'_{1}(\omega_j+\omega_j^{p-2})  +\dots + u'_{\frac{p-3}{2}}(\omega_j^{\frac{p-3}{2}}+\omega_j^{\frac{p+1}{2}})+u'_{\frac{p-1}{2}}(\omega_j^{\frac{p-1}{2}})\\
&=& \displaystyle u'_0+2u'_{1}\cos(\frac{1 \cdot 2 \pi j}{p-1}) +\dots + 2u'_{\frac{p-3}{2}}\cos(\frac{\frac{p-3}{2} \cdot 2 \pi j}{p-1}) \\ &&+u'_{\frac{p-1}{2}}\cos(\frac{\frac{p-1}{2} \cdot 2 \pi j}{p-1}).
\end{array}
\]
 $\lambda_j = \lambda_{p-1-j}$ since $\cos(x)$ is an even function, thus it follows that all eigenvalues are nonnegative is equivalent to that
 \[
 [\cos(\frac{0 \cdot 2 \pi j}{p-1}), \dots,\cos(\frac{\frac{p-1}{2} \cdot 2 \pi j}{p-1})]u \geq 0
 \]
  for $j=1,\dots,\frac{p-1}{2}$, which is equivalent to the condition $Cu \geq 0$ in Problem (b). 

The eigenvalue corresponding to to the eigenvector $[1,\dots,1]$ in $U''$ is $\lambda_0=\sum_i^{p-1}u'_i$. We get the maximal value of $b$ when $u_i-u_i'$ is maximized, that is when $\lambda_0=0$ and thus
\[
b= \frac{[1,\dots,1]u}{p-1}=\frac{u_0+2u_1+\dots+2u_{(p-3)/2}+u_{(p-1)/2}}{p-1}.
\]

\end{proof}

\begin{proof}[Proof of Theorem \ref{thm:density2}]
Lemma \ref{lem:density1} shows that problems (a) and (b') are equivalent and Lemma \ref{lem:density2} shows that problems (b') and (b) are equivalent. It immediately follows that problems (a) and (b) are equivalent.
\end{proof}

\begin{proof}[Proof of Corollary \ref{cor:density}]
Let 
\[
B=\min \{D \in \mathbb{Z_+} : \lambda_p(D) > 0 \}.
\]
To show that
\[
B \leq \frac{p+3}{2}
\]
let us use the fact that the algebraic certificate in Theorem \ref{thm:density} is on the desired form. We have 
\[
\begin{array}{rl}
B  &\leq \displaystyle \min \{D \in \mathbb{Z_+} : \frac{D^3 - (\frac{p+3}{2})D^2+(\frac{p+3}{2}-1)D}{p-1} > 0 \} \\
&= \displaystyle \min \{D \in \mathbb{Z_+} : \frac{D(D-1)(D-\frac{p+1}{2})}{p-1} > 0 \} \\
&= \displaystyle  \frac{p+3}{2}.
\end{array}
\]
On the other hand, to show that 
\[
B \geq \left\lceil \frac{p+3}{4} \right\rceil
\]
note that by Theorem \ref{thm:density2} we have
\[
B = \min \{D \in \mathbb{Z_+} : \frac{u_+}{p-1}(D - 1 -\frac{u_0}{u_+}(p-1)   )D(D-1) \}.
\]
By the matrix condition $Vu=v$ we know that the sum of the $u_i$ with $i>0$ which do not contribute to an arithmetic progression is zero. Furthermore, if $t$ and $s$ are the unique integer such that $r^t=2$ and $r^s=\frac{p-1}{2}$, then $u_t$s orbit contributes to $4$ arithmetic progressions and $u_s$s orbit contributes to $2$ arithmetic progressions respectively. In other words $4u_t+2u_s=1$, and hence
\[
u_+ =u_0+ 2u_t+u_s = u_0 + \frac{1}{2},
\]
so
\[
B =  \min \{D \in \mathbb{Z_+} :\frac{u_+}{p-1}(D - 1 -\frac{u_0}{u_0 + \frac{1}{2}}(p-1)   )D(D-1) \}.
\]
To find a lower bound to $B$ we should find a lower bound for $\frac{u_0}{u_0 + \frac{1}{2}}$. Since $u_0 \geq \max_{i=1}^{p-1} u_i$ in order for $U$ to be positive semidefinite and since $4u_t+2u_s=1$ we have $u_0 \geq \frac{1}{6}$. Hence
\[
\begin{array}{rl}
\displaystyle B \geq& \displaystyle \min \{D \in \mathbb{Z_+} :\frac{u_+}{p-1}(D - 1 -\frac{\frac{1}{6}}{\frac{1}{6} + \frac{3}{6}}(p-1)   )D(D-1) \}. \\
 =& \displaystyle  \lceil 1 + \frac{1}{4}(p-1) \rceil
 \end{array}
\]

\end{proof}
\section{Additional numerical results}
\label{sec:numerical2}
We can easily generate the list of all necklaces of length $n$ with $n \delta$ ones using the recursive formula introduced in Section \ref{sec:graycode}. We can use the list to find $W(k,\mathbb{Z}_n,\delta)$, we merely need to count the number of arithmetic progressions in each necklace and find the minimum among these. Note that $n$ does not have to be prime as long as we are cautious when we count arithmetic progressions to avoid double counting and degenerate arithmetic progressions.

We have found exact solutions using fixed density necklaces for $W(k,\mathbb{Z}_n,\delta)$ for $3 \leq k \leq 5$, $ 5\leq n \leq 32$ and $\delta \in \{0,\frac{1}{n},\dots,\frac{n-1}{n},1\}$. To find the exact solution we used the code for finding fixed density necklaces by Sawada \cite{Sawada2012b}, and code implemented in matlab by the author for counting the number of arithmetic progressions in every necklace. The code was implemented on the Triton computer cluster, which is part of the Aalto Science-IT project, and we used 100 computer nodes for 5 days to carry out the calculations. The results can be found in Tables \ref{tab:bubble3_1}, \ref{tab:bubble3_2}, \ref{tab:bubble4_1}, \ref{tab:bubble4_2}, \ref{tab:bubble5_1} and \ref{tab:bubble5_2} in the appendix.

Using the degree 3 relaxation of the Lasserre hierarchy we have found lower bounds to $W(3,\mathbb{Z}_p,\delta)$ where $p \leq 300$ is a prime and 
\[
\delta \in \{0, \frac{1}{20p}, \frac{2}{20p}, \dots, \frac{20p-1}{20p}, 1\}.
\]
 For $300 \leq p \leq 613$ we start to run into numerical problems for larger $\delta$, which might be because of our particular choice of solver, so then we have fewer and perhaps less accurate data points. The results for $\delta < \frac{1}{3}$ seems accurate in this interval, so we feel confident using this data for analysis on the asymptotic behavior of the degree 3 relaxation.

Using the degree 5 relaxation of the Lasserre hierarchy we have found lower bounds to $W(k,\mathbb{Z}_p,\delta)$ for primes $p \leq 19$, $k=3,4,5$ and  
\[
\delta \in \{0,\frac{1}{20p}, \frac{2}{20p}, \dots,\frac{20p-1}{20p}, 1\}.
\]

In Figure \ref{fig:minAP1} we show the results for $k=3$ and $D=\delta p \in \{0,\dots,17\}$ when $p=17$, for which we have numerical data for all different relaxations we make, as well as a lower bound from Theorem \ref{thm:smallprimes}. In Figure \ref{fig:minAP2} we showed a zoomed in version around where $W(3,\mathbb{Z}_{17},\delta)$ becomes positive. For small values of $D$ and $p$ we can find algebraic solutions to check when the points coincide, unfortunately this is generally not the case for the larger values of $D$ and $p$ which we are more interested in. The case $p=17$ is a good representative for the situation for all small primes, and hence the numerical certificates as well as the algebraic certificates one can get from Theorem \ref{thm:density2} work very well here. The problem is that the larger $p$ is the worse the relaxations will get. We discuss this further in Section \ref{sec:discussion}.

\begin{figure}
\includegraphics{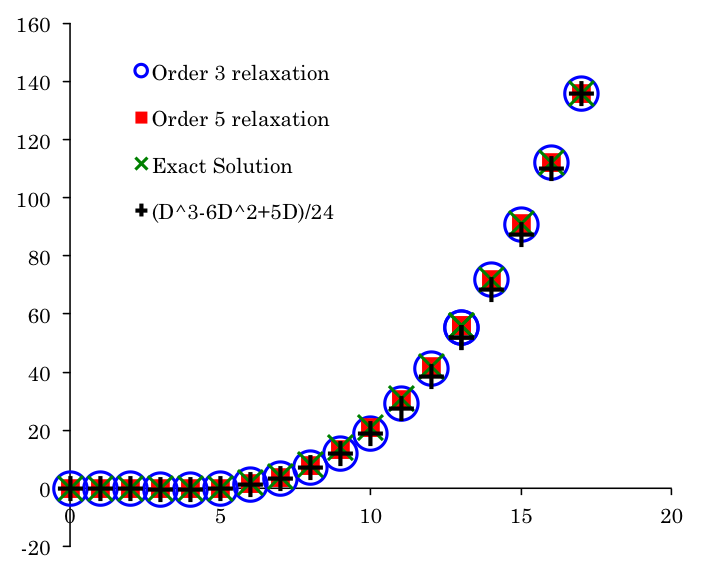}
\caption{The minimum number of arithmetic progressions of length 3 as a function of the density in $\mathbb{Z}_{17}$.}
\label{fig:minAP1}
\end{figure}

\begin{figure}
\includegraphics{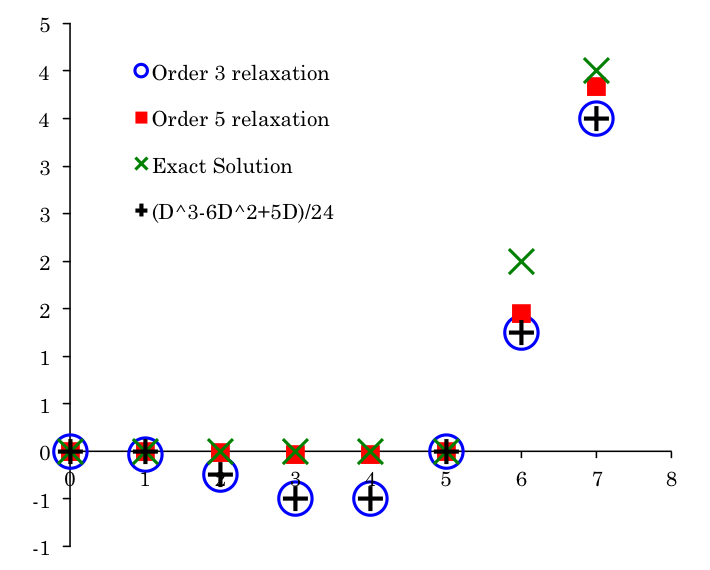}
\caption{A zoomed in version of Figure \ref{fig:minAP1}.}
\label{fig:minAP2}
\end{figure}

We also provide a list of one necklace per $D$ that achieve the lower bound as a certificate in Tables \ref{tab:smallaces1} and \ref{tab:smallaces2}. There is recent research on how to find necklaces with less arithmetic progressions than average for higher $p$ \cite{Butler2010,Lu2012}. In both papers they analyze the problem of avoiding arithmetic progressions in a $2$-coloring of $[n]$, which is closely related to the fixed-density case, and it is likely that similar constructions are applicable here.

\begin{table}
\begin{center}
\begin{tabular}{ccc}
{\bf D} & {\bf A.P.s} & {\bf smallest necklace} \\
0 & 0 & 0000000 \\
1 & 0 & 0000001 \\
2 & 0 & 0000011 \\
3 & 0 & 0001011 \\
4 & 2 & 0001111 \\
5 & 6 & 0011111 \\
6 & 12 & 0111111 \\
7 & 21 & 1111111 \\
\end{tabular}
\caption{Necklaces that contains the least number of arithmetic progressions for each $D$ in $\mathbb{Z}_{7}$.}
\label{tab:smallaces1}
\end{center}
\end{table}

\begin{table}
\begin{center}
\begin{tabular}{ccc}
{\bf D} & {\bf A.P.s} & {\bf smallest necklace} \\
0 & 0 & 0000000000000000000000000000000 \\
1 & 0 & 0000000000000000000000000000001 \\
2 & 0 & 0000000000000000000000000000011 \\
3 & 0 & 0000000000000000000000000001011 \\
4 & 0 & 0000000000000000000000000011011 \\
5 & 0 & 0000000000000000000000101100011 \\
6 & 0 & 0000000000000000000010110001011 \\
7 & 0 & 0000000000000000001011000011011 \\
8 & 0 & 0000000000000000011011000011011 \\
9 & 2 & 0000000000000001011010001011011 \\
10 & 5 & 0000000000000001101101000110111 \\
11 & 9 & 0000000000000101011110001011011 \\
12 & 14 & 0000000000000001111110000111111 \\
13 & 19 & 0000000001010011110001101110011 \\
14 & 27 & 0000000001101011110001011011011 \\
15 & 37 & 0000000000000111111100011111111 \\
16 & 48 & 0000000000000111111110011111111 \\
17 & 61 & 0000000000001111111100111111111 \\
18 & 76 & 0000000000001111111101111111111 \\
19 & 92 & 0000000000011111111101111111111 \\
20 & 110 & 0000000000011111111111111111111 \\
21 & 130 & 0000000000111111111111111111111 \\
22 & 152 & 0000000001111111111111111111111 \\
23 & 177 & 0000000011111111111111111111111 \\
24 & 204 & 0000000111111111111111111111111 \\
25 & 234 & 0000001111111111111111111111111 \\
26 & 266 & 0000011111111111111111111111111 \\
27 & 301 & 0000111111111111111111111111111 \\
28 & 338 & 0001111111111111111111111111111 \\
29 & 378 & 0011111111111111111111111111111 \\
30 & 420 & 0111111111111111111111111111111 \\
31 & 465 & 1111111111111111111111111111111
\end{tabular}
\caption{Necklaces that contains the least number of arithmetic progressions for each $D$ in $\mathbb{Z}_{31}$.}
\label{tab:smallaces2}
\end{center}
\end{table}

 It is shown in the first paper that there exists colorings in which the colors come in large blocks that provide much better bounds than random colorings, and in the second paper it is shown that periodic colorings can provide even better bounds. For computational reasons we are not able to do examples for $p$ high enough to make reasonable conjectures how fixed-density sets avoiding arithmetic progressions look like, and we have not tried to apply their methods. Note that finding a class of fixed-density necklaces does not help to imply a new proof of Szemer\'edi's theorem directly as we would need a certificate that all fixed-density necklaces have a positive density. Somewhat surprisingly though we have in some cases seen correlations between good classes of colorings and coefficients in the corresponding semidefinite programs for finding a certificate. Finding a good way of coloring groups could thus prove to be useful for finding good SDP relaxations.
 
Another opportunity to analyze what happens as $p$ grows is to look at the distributions of the number of arithmetic progressions for all different necklaces of fixed $p$, $k$ and $\delta$. When for example $p=7$, $k=3$ and $\delta=4/7$ there are $3$ necklaces with $2$ arithmetic progressions, and $2$ necklaces with $3$, clearly $W(7,3,4/7)=2$ and we have the distribution in Figure \ref{fig:allAP1}. More interestingly, for $p=31$, $k=3$ and $\delta=16/31$ we have the situation in Figure \ref{fig:allAP2}; it looks like finding the minimum number of arithmetic progressions can be approximated by analyzing the left tail of a probability distribution. The distribution looks very similar as we vary $p$ and $\delta$, it just gets shifted and has a different height. The end goal is to go from bounds on $\mathbb{Z}_p$ for all prime numbers $p$ to $[n]$ for all integers $n$, and using a theorem similar to Corollary \ref{cor:limit} to get a new proof for Szemer\'edi's theorem. Hence if it would be possible to prove that for any density $\delta > 0$ there exists an $N$ such that for all $p \geq N$ it holds that the left-most non-zero value of the distribution of the number of arithmetic progressions in different necklaces of length $p$ and density $\delta$ is positive it is likely possible that a proof would follow. It is unclear whether the tools for approximating the tail are strong enough to find a new proof of Szemer\'edi's theorem, and even if the tools are sufficient it would most probably be highly non-trivial to find a new proof.

\begin{figure}
\includegraphics{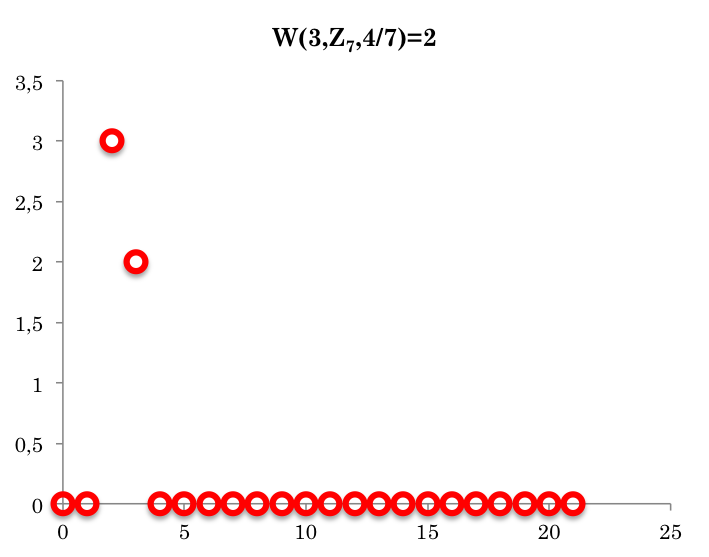}
\caption{The plot shows how many necklaces of length $7$ and density $4/7$ contain $n$ 3-arithmetic progressions for $n=0,\dots,21$.}
\label{fig:allAP1}
\end{figure}

\begin{figure}
\includegraphics{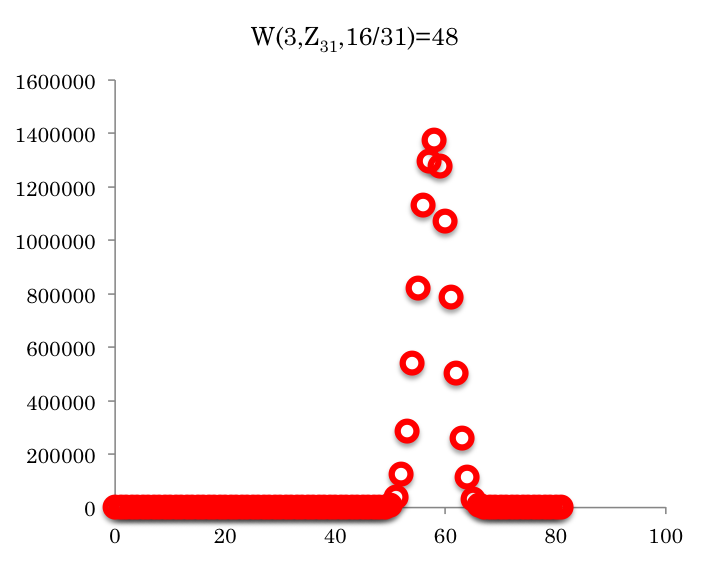}
\caption{The plot shows how many necklaces of length $31$ and density $16/31$ contain $n$ 3-arithmetic progressions for $n=0,\dots,81$.}
\label{fig:allAP2}
\end{figure}

\section{Discussion}
\label{sec:discussion}

\subsection{How Szemer\'edi's theorem could perhaps be generalized for arithmetic progressions of length $3$}
In order to prove Szemer\'edi's theorem by providing bounds 
\[
\lambda(k,\mathbb{Z}_n,\delta) \leq W(k,\mathbb{Z}_n,\delta)
\]
 it is by Corollary \ref{cor:limit}, and sufficiently good tools to translate results between $\mathbb{Z}_n$ and $[n]$, necessary and sufficient to find bounds as sharp as 
\[
\delta_*(k) = \lim_{n \rightarrow \infty} \min \{ \delta : \lambda(k,\mathbb{Z}_n,\delta) > 0\} = 0.
\]
Already finding sharp enough bounds in the case $k=3$ would be a great result. Depending on how fast the convergence is, such bound might additionally improve the current best existence bounds for arithmetic progressions by Gowers \cite{Gowers2001}. It can be shown by a probabilistic argument that $W(k,\mathbb{Z}_n,\delta)$ with fixed $k$ and $n$ behaves as a degree $k$ polynomial when $\delta$ is large enough, but it is not clear whether this should also be the case for small $\delta$. Any additional understanding on how $W(k,\mathbb{Z}_n,\delta)$ looks like for small $\delta$ could possibly benefit in a better understanding on how to restrict polynomials in a relaxation of Putinar's positivstellensatz in the search for an improved lower bound $\lambda(k,\mathbb{Z}_n,\delta)$.

In Figure \ref{fig:allbounds} we have in the same plot put the upper and lower bounds from Conjecture \ref{cor:density} as well as numerical upper bounds from the degree 3 and 5 relaxations of Putinar's positivstellensatz. All methods presented perform very well for small $p$. In the first few cases we can find algebraic solutions verifying that we find the exact solution with our approximation. It seems like using the degree 3 relaxation in its full generality rather than our simplification in Theorem \ref{thm:density} does not make a huge difference, and it seems like we need to increase the degree to make sure $\delta_*$ converge to $0$. We have too few data points of the degree 5 relaxation to make any reasonable conjectures about the convergence of those upper bounds. We suggest as a next step to try to come up with reasonable simplification of the degree 5 relaxation that would allow us to find more data points. How to find general patterns for a simplified degree $5$ relaxation based on numerical data is highly non-trivial, and it is unclear whether a degree $5$ relaxation is sufficient to generalize Szemer\'edi's theorem for $k=3$, or if one will find lower bounds as in Corollary \ref{cor:density}. We know in theory by Putinar's positivstellensatz that if the relaxation is of a degree high enough we will find a quantitative version of Szmer\'edi's theorem, but in practice we have to either limit the degree of the relaxation or add restrictions to get a simpler problem but worse lower bound. It is thus unlikely that the theoretically optimal lower bound will ever be obtained, but one cannot exclude the possibility that a lower bound strong enough to obtain a quantitative version of Szemer\'edi's theorem can be obtained. As we have found algebraic bounds with polynomials of degree 3, it is likely that improved bounds with polynomials of higher degree look similar, or at least share some similarities, with the polynomials of degree 3. Another possibility would thus be to try to find a certificate based on polynomials of high degree with a lot of structure that are similar to the polynomials in Theorem \ref{thm:density}. 

\begin{figure}
\includegraphics{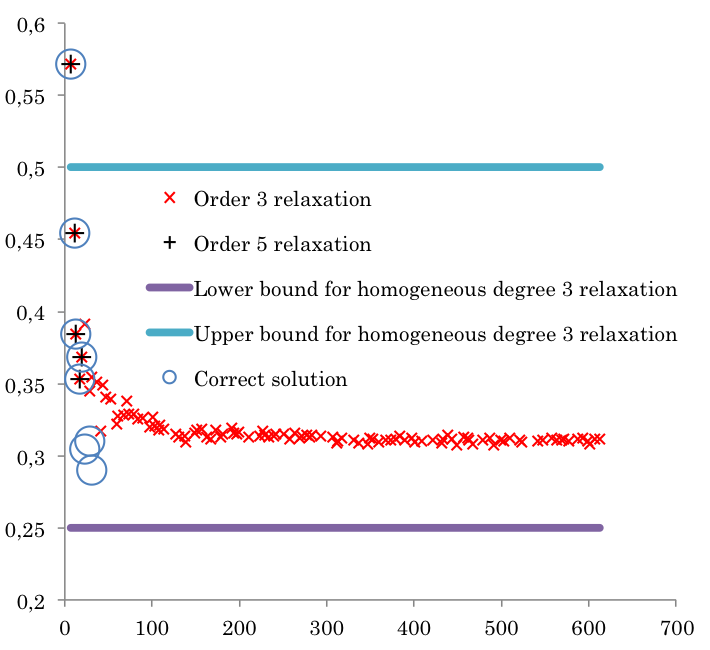}
\caption{Lower and upper bounds from Conjecture \ref{cor:density}, numerical upper bounds from degree 3 and 5 relaxations of Putinar's Positivstellensatz and exact results found using fixed density necklaces.}
\label{fig:allbounds}
\end{figure}

\subsection{Longer arithmetic progressions}

Recall that  $W(k,G,\delta)$ denotes the minimal number of arithmetic progressions in the subset $S \subseteq G$ with $|S| = \delta |G|$. It is easy to see that
\[ 
W(k,G,D/p) = \min \{\sum_{\{a_1,\dots,a_k\} \textrm{ is an A.P. in } G} x_{a_1}\cdots x_{a_k} : x_i \in \{0,1\} , \sum_{i=0}^{p-1} x_i = D\},
\]
and using the methods introduced in Sections \ref{sec:poly} and \ref{sec:Sym} we make the following relaxation to get an optimization problem which if it is solved gives an algebraic certificate for the lower bound:
\[ 
W(3,G,D/|G|) \geq \max \{ \lambda : \sum_{\{a_1,\dots,a_k\} \textrm{ A.P. in } G} X_{a_1}\cdots X_{a_k} - \lambda = S, X \in K \} 
\]
where $K = \{X \in [0,1]^n: \sum_{i=0}^{|G|-1}X_i = D\}$ and $S = \sum_i S_i^2g_i$ for $S_i,g_i \in \mathbb{R}[X_1,\dots,X_k]$ where $g_i$ are the half-spaces defining $K$.

Since $\sum_{i=0}^{|G|-1} x_i = D$ on $\{-1,1\}^n$ it follows that for example $\sum_{i=0}^{|G|-1} x_i^t - D= 0$ for all $t \in \mathbb{Z}_+$, $\sum_{i=0}^{|G|-1} x_i^sx_j^t - D(D-1) = 0$ for all $s \neq t \in \mathbb{Z}_+$, and so on, providing several new possibilities for polynomials that are nonnegative on $\{-1,1\}^n$. Let us denote the set of all possible conditions of this form $\{R_0,R_1,\dots,\}$.

These redundant conditions simplifies the algebraic certificates we get from Putinar's Positivstellensatz, and replacing some of the original $g_i$s to these new constraints could possibly simplify the computations significantly. For any $k$ we can write $S$ on the form 
\[
S=\sum_ia_iS_{i}^2g_i + \sum_j b_jR_j +c,
\]
where $a_i \in \{0,1\}$ are used to simplify the problem, $b_j \in \mathbb{R}$ are constants indicating how much of the redundant condition $R_j$ should be added and $c$ a constant. This is one option to move forward with the suggested methods as the redundant conditions add much more freedom to the problem. As a generalization of Szemer\'edi's theorem would be extremely difficult to accomplish we want to stress that these are just speculations on how to improve the results in this article towards that goal. Even though there are strong theoretical results in real algebraic geometry, the current methods for applying the methods in practice have their limitations. If even possible probably many years of further research in this direction is required to find a quantitative version of Szemer\'edi's theorem.

\section*{Acknowledgements}
I would like to thank Alexander Engstr\"om for introducing me to Szemer\'edi's theorem and for suggesting how the problem could be approached. Additionally I would want to thank Markus Schweighofer and Cynthia Vinzant for their valuable feedback.

\newpage
\section{Appendix}
\newpage

\begin{table}
\begin{center}
\begin{tabular}{c|ccccc|ccccc|ccccc|ccccc}
n $\backslash$ D & 0 & 1 & 2 & 3 & 4 & 5 & 6 & 7 & 8 & 9 & 10 & 11 & 12 & 13 & 14 & 15 & 16 & 17   \\
\hline
5 & 0 & 0 & 0 & 1 & 4 & 10 &  &  &  &  &  &  &  &  &  &  &  &  \\
6 & 0 & 0 & 0 & 0 & 0 & 4 & 8 &  &  &  &  &  &  &  &  &  &  &  \\
7 & 0 & 0 & 0 & 0 & 2 & 6 & 12 & 21 &  &  &  &  &  &  &  &  &  &  \\
8 & 0 & 0 & 0 & 0 & 0 & 3 & 8 & 15 & 28 &  &  &  &  &  &  &  &  &  \\
9 & 0 & 0 & 0 & 0 & 0 & 1 & 2 & 11 & 20 & 30 &  &  &  &  &  &  &  &  \\
\hline
10 & 0 & 0 & 0 & 0 & 0 & 2 & 4 & 10 & 16 & 28 & 45 &  &  &  &  &  &  &  \\
11 & 0 & 0 & 0 & 0 & 0 & 2 & 5 & 11 & 18 & 28 & 40 & 55 &  &  &  &  &  &  \\
12 & 0 & 0 & 0 & 0 & 0 & 1 & 2 & 5 & 8 & 18 & 28 & 39 & 52 &  &  &  &  &  \\
13 & 0 & 0 & 0 & 0 & 0 & 1 & 4 & 8 & 14 & 22 & 32 & 45 & 60 & 78 &  &  &  &  \\
14 & 0 & 0 & 0 & 0 & 0 & 0 & 0 & 4 & 8 & 16 & 24 & 36 & 48 & 66 & 91 &  &  &  \\
\hline
15 & 0 & 0 & 0 & 0 & 0 & 1 & 2 & 5 & 8 & 12 & 20 & 30 & 40 & 58 & 76 & 95 &  &  \\
16 & 0 & 0 & 0 & 0 & 0 & 0 & 0 & 1 & 4 & 11 & 20 & 29 & 40 & 55 & 72 & 91 & 120 &  \\
17 & 0 & 0 & 0 & 0 & 0 & 0 & 2 & 4 & 8 & 14 & 21 & 31 & 42 & 56 & 72 & 91 & 112 & 136 \\
18 & 0 & 0 & 0 & 0 & 0 & 0 & 0 & 0 & 0 & 4 & 8 & 12 & 16 & 34 & 52 & 70 & 88 & 110 \\
19 & 0 & 0 & 0 & 0 & 0 & 0 & 0 & 3 & 6 & 11 & 18 & 26 & 36 & 48 & 62 & 79 & 98 & 120 \\
\hline
20 & 0 & 0 & 0 & 0 & 0 & 0 & 0 & 0 & 0 & 5 & 11 & 19 & 28 & 39 & 51 & 65 & 80 & 104 \\
21 & 0 & 0 & 0 & 0 & 0 & 0 & 0 & 1 & 2 & 3 & 9 & 15 & 22 & 33 & 42 & 59 & 77 & 95 \\
22 & 0 & 0 & 0 & 0 & 0 & 0 & 0 & 0 & 0 & 4 & 8 & 14 & 20 & 32 & 44 & 58 & 72 & 92 \\
23 & 0 & 0 & 0 & 0 & 0 & 0 & 0 & 1 & 3 & 7 & 12 & 18 & 26 & 36 & 47 & 61 & 76 & 94 \\
24 & 0 & 0 & 0 & 0 & 0 & 0 & 0 & 0 & 0 & 1 & 2 & 6 & 10 & 20 & 30 & 39 & 48 & 70 \\
\hline
25 & 0 & 0 & 0 & 0 & 0 & 0 & 0 & 0 & 2 & 6 & 10 & 15 & 22 & 31 & 41 & 53 & 66 & 82 \\
26 & 0 & 0 & 0 & 0 & 0 & 0 & 0 & 0 & 0 & 2 & 4 & 10 & 16 & 24 & 32 & 44 & 56 & 72 \\
27 & 0 & 0 & 0 & 0 & 0 & 0 & 0 & 0 & 0 & 3 & 6 & 9 & 12 & 21 & 30 & 39 & 48 & 60 \\
28 & 0 & 0 & 0 & 0 & 0 & 0 & 0 & 0 & 0 & 2 & 4 & 7 & 12 & 20 & 28 & 37 & 48 & 63 \\
29 & 0 & 0 & 0 & 0 & 0 & 0 & 0 & 0 & 0 & 3 & 6 & 11 & 16 & 23 & 32 & 41 & 52 & 67 \\
\hline
30 & 0 & 0 & 0 & 0 & 0 & 0 & 0 & 0 & 0 & 1 & 4 & 6 & 8 & 14 & 20 & 26 & 32 & 44 \\
31 & 0 & 0 & 0 & 0 & 0 & 0 & 0 & 0 & 0 & 2 & 5 & 9 & 14 & 19 & 27 & 37 & 48 & 61 \\
32 & 0 & 0 & 0 & 0 & 0 & 0 & 0 & 0 & 0 & 0 & 2 & 5 & 8 & 11 & 16 & 23 & 32 & 46 \\
\end{tabular}
\caption{$W(3,\mathbb{Z}_n,D/n)$ for different $n$ and $D$.}
\label{tab:bubble3_1}
\end{center}
\end{table}

\newpage

\begin{table}
\begin{center}
\begin{tabular}{c|ccccc|ccccc|ccccc|ccccc}
n $\backslash$ D & 18 & 19 & 20 & 21 & 22 & 23 & 24 & 25 & 26 & 27& 28 & 29 & 30 & 31 & 32    \\
\hline
5 &  &  &  &  &  &  &  &  &  &  &  &  &  &  &  \\
6 &  &  &  &  &  &  &  &  &  &  &  &  &  &  &  \\
7 &  &  &  &  &  &  &  &  &  &  &  &  &  &  &  \\
8 &  &  &  &  &  &  &  &  &  &  &  &  &  &  &  \\
9 &  &  &  &  &  &  &  &  &  &  &  &  &  &  &  \\
\hline
10 &  &  &  &  &  &  &  &  &  &  &  &  &  &  &  \\
11 &  &  &  &  &  &  &  &  &  &  &  &  &  &  &  \\
12 &  &  &  &  &  &  &  &  &  &  &  &  &  &  &  \\
13 &  &  &  &  &  &  &  &  &  &  &  &  &  &  &  \\
14 &  &  &  &  &  &  &  &  &  &  &  &  &  &  &  \\
\hline
15 &  &  &  &  &  &  &  &  &  &  &  &  &  &  &  \\
16 &  &  &  &  &  &  &  &  &  &  &  &  &  &  &  \\
17 &  &  &  &  &  &  &  &  &  &  &  &  &  &  &  \\
18 & 132 &  &  &  &  &  &  &  &  &  &  &  &  &  &  \\
19 & 144 & 171 &  &  &  &  &  &  &  &  &  &  &  &  &  \\
\hline
20 & 128 & 153 & 190 &  &  &  &  &  &  &  &  &  &  &  &  \\
21 & 114 & 141 & 168 & 196 &  &  &  &  &  &  &  &  &  &  &  \\
22 & 112 & 136 & 160 & 190 & 231 &  &  &  &  &  &  &  &  &  &  \\
23 & 114 & 137 & 162 & 190 & 220 & 253 &  &  &  &  &  &  &  &  &  \\
24 & 90 & 112 & 135 & 160 & 188 & 217 & 248 &  &  &  &  &  &  &  &  \\
\hline
25 & 100 & 119 & 140 & 170 & 200 & 231 & 264 & 300 &  &  &  &  &  &  &  \\
26 & 88 & 108 & 128 & 154 & 180 & 210 & 240 & 276 & 325 &  &  &  &  &  &  \\
27 & 72 & 99 & 126 & 153 & 180 & 210 & 240 & 276 & 312 & 351 &  &  &  &  &  \\
28 & 79 & 97 & 116 & 139 & 163 & 189 & 216 & 252 & 288 & 325 & 378 &  &  &  &  \\
29 & 83 & 101 & 120 & 142 & 166 & 193 & 222 & 254 & 288 & 325 & 364 & 406 &  &  &  \\
\hline
30 & 56 & 68 & 80 & 110 & 140 & 168 & 192 & 228 & 264 & 300 & 336 & 378 & 435 &  &  \\
31 & 76 & 92 & 110 & 130 & 152 & 177 & 204 & 234 & 266 & 301 & 338 & 378 & 420 & 465 &  \\
32 & 60 & 78 & 95 & 115 & 135 & 157 & 180 & 211 & 244 & 277 & 312 & 351 & 392 & 435 & 496 \\
\end{tabular}
\caption{$W(3,\mathbb{Z}_n,D/n)$ for different $n$ and $D$.}
\label{tab:bubble3_2}
\end{center}
\end{table}

\newpage 

\begin{table}
\begin{center}
\begin{tabular}{c|ccccc|ccccc|ccccc|ccccc}
n $\backslash$ D & 0 & 1 & 2 & 3 & 4 & 5 & 6 & 7 & 8 & 9 & 10 & 11 & 12 & 13 & 14 & 15 & 16 & 17   \\
\hline
5 & 0 & 0 & 0 & 0 & 2 & 10 &  &  &  &  &  &  &  &  &  &  &  &  \\
6 & 0 & 0 & 0 & 0 & 0 & 4 & 8 &  &  &  &  &  &  &  &  &  &  &  \\
7 & 0 & 0 & 0 & 0 & 0 & 3 & 9 & 21 &  &  &  &  &  &  &  &  &  &  \\
8 & 0 & 0 & 0 & 0 & 0 & 0 & 4 & 12 & 28 &  &  &  &  &  &  &  &  &  \\
9 & 0 & 0 & 0 & 0 & 0 & 1 & 2 & 11 & 20 & 30 &  &  &  &  &  &  &  &  \\
\hline
10 & 0 & 0 & 0 & 0 & 0 & 0 & 0 & 4 & 8 & 24 & 45 &  &  &  &  &  &  &  \\
11 & 0 & 0 & 0 & 0 & 0 & 0 & 0 & 5 & 11 & 21 & 35 & 55 &  &  &  &  &  &  \\
12 & 0 & 0 & 0 & 0 & 0 & 1 & 2 & 5 & 8 & 18 & 28 & 39 & 52 &  &  &  &  &  \\
13 & 0 & 0 & 0 & 0 & 0 & 0 & 0 & 2 & 6 & 13 & 22 & 36 & 54 & 78 &  &  &  &  \\
14 & 0 & 0 & 0 & 0 & 0 & 0 & 0 & 0 & 0 & 6 & 12 & 24 & 36 & 60 & 91 &  &  &  \\
\hline
15 & 0 & 0 & 0 & 0 & 0 & 1 & 2 & 5 & 8 & 12 & 20 & 30 & 40 & 58 & 76 & 95 &  &  \\
16 & 0 & 0 & 0 & 0 & 0 & 0 & 0 & 0 & 0 & 3 & 8 & 16 & 24 & 40 & 60 & 84 & 120 &  \\
17 & 0 & 0 & 0 & 0 & 0 & 0 & 0 & 0 & 0 & 4 & 8 & 15 & 25 & 39 & 56 & 78 & 104 & 136 \\
18 & 0 & 0 & 0 & 0 & 0 & 0 & 0 & 0 & 0 & 4 & 8 & 12 & 16 & 34 & 52 & 70 & 88 & 110 \\
19 & 0 & 0 & 0 & 0 & 0 & 0 & 0 & 0 & 0 & 2 & 5 & 9 & 15 & 27 & 40 & 58 & 79 & 105 \\
\hline
20 & 0 & 0 & 0 & 0 & 0 & 0 & 0 & 0 & 0 & 0 & 0 & 4 & 8 & 16 & 25 & 36 & 48 & 80 \\
21 & 0 & 0 & 0 & 0 & 0 & 0 & 0 & 1 & 2 & 3 & 9 & 15 & 22 & 33 & 42 & 59 & 77 & 95 \\
22 & 0 & 0 & 0 & 0 & 0 & 0 & 0 & 0 & 0 & 0 & 0 & 0 & 0 & 10 & 20 & 32 & 44 & 64 \\
23 & 0 & 0 & 0 & 0 & 0 & 0 & 0 & 0 & 0 & 0 & 1 & 4 & 7 & 13 & 20 & 31 & 45 & 63 \\
24 & 0 & 0 & 0 & 0 & 0 & 0 & 0 & 0 & 0 & 1 & 2 & 6 & 10 & 20 & 30 & 39 & 48 & 70 \\
\hline
25 & 0 & 0 & 0 & 0 & 0 & 0 & 0 & 0 & 0 & 0 & 0 & 2 & 4 & 8 & 14 & 20 & 27 & 41 \\
26 & 0 & 0 & 0 & 0 & 0 & 0 & 0 & 0 & 0 & 0 & 0 & 0 & 0 & 4 & 8 & 16 & 24 & 38 \\
27 & 0 & 0 & 0 & 0 & 0 & 0 & 0 & 0 & 0 & 0 & 0 & 0 & 3 & 6 & 11 & 17 & 26 & 38 \\
28 & 0 & 0 & 0 & 0 & 0 & 0 & 0 & 0 & 0 & 0 & 0 & 0 & 0 & 0 & 3 & 9 & 16 & 26 \\
29 & 0 & 0 & 0 & 0 & 0 & 0 & 0 & 0 & 0 & 0 & 0 & 0 & 1 & 3 & 7 & 13 & 20 & 28 \\
\hline
30 & 0 & 0 & 0 & 0 & 0 & 0 & 0 & 0 & 0 & 0 & 0 & 0 & 0 & 0 & 0 & 4 & 8 & 17 \\
31 & 0 & 0 & 0 & 0 & 0 & 0 & 0 & 0 & 0 & 0 & 0 & 0 & 0 & 2 & 5 & 10 & 15 & 23 \\
32 & 0 & 0 & 0 & 0 & 0 & 0 & 0 & 0 & 0 & 0 & 0 & 0 & 0 & 0 & 1 & 4 & 9 & 15 \\
\end{tabular}
\caption{$W(4,\mathbb{Z}_n,D/n)$ for different $n$ and $D$.}
\label{tab:bubble4_1}
\end{center}
\end{table}

\newpage

\begin{table}
\begin{center}
\begin{tabular}{c|ccccc|ccccc|ccccc|ccccc}
n $\backslash$ D & 18 & 19 & 20 & 21 & 22 & 23 & 24 & 25 & 26 & 27& 28 & 29 & 30 & 31 & 32    \\
\hline
5 &  &  &  &  &  &  &  &  &  &  &  &  &  &  & \\
6 &  &  &  &  &  &  &  &  &  &  &  &  &  &  & \\ 
7 &  &  &  &  &  &  &  &  &  &  &  &  &  &  &  \\
8 &  &  &  &  &  &  &  &  &  &  &  &  &  &  &  \\
9 &  &  &  &  &  &  &  &  &  &  &  &  &  &  &  \\
\hline
10 &  &  &  &  &  &  &  &  &  &  &  &  &  &  &  \\
11 &  &  &  &  &  &  &  &  &  &  &  &  &  &  &  \\
12 &  &  &  &  &  &  &  &  &  &  &  &  &  &  &  \\
13 &  &  &  &  &  &  &  &  &  &  &  &  &  &  &  \\
14 &  &  &  &  &  &  &  &  &  &  &  &  &  &  &  \\
\hline
15 &  &  &  &  &  &  &  &  &  &  &  &  &  &  &  \\
16 &  &  &  &  &  &  &  &  &  &  &  &  &  &  &  \\
17 &  &  &  &  &  &  &  &  &  &  &  &  &  &  &  \\
18 & 132 &  &  &  &  &  &  &  &  &  &  &  &  &  &  \\
19 & 135 & 171 &  &  &  &  &  &  &  &  &  &  &  &  &  \\
\hline
20 & 112 & 144 & 190 &  &  &  &  &  &  &  &  &  &  &  &  \\
21 & 114 & 141 & 168 & 196 &  &  &  &  &  &  &  &  &  &  &  \\
22 & 84 & 112 & 140 & 180 & 231 &  &  &  &  &  &  &  &  &  &  \\
23 & 82 & 108 & 137 & 171 & 209 & 253 &  &  &  &  &  &  &  &  &  \\
24 & 90 & 112 & 135 & 160 & 188 & 217 & 248 &  &  &  &  &  &  &  &  \\
\hline
25 & 56 & 72 & 90 & 130 & 170 & 210 & 252 & 300 &  &  &  &  &  &  &  \\
26 & 52 & 70 & 88 & 116 & 144 & 180 & 216 & 264 & 325 &  &  &  &  &  &  \\
27 & 52 & 70 & 90 & 114 & 142 & 175 & 211 & 254 & 300 & 351 &  &  &  &  &  \\
28 & 37 & 52 & 68 & 91 & 114 & 140 & 168 & 216 & 264 & 312 & 378 &  &  &  &  \\
29 & 40 & 53 & 71 & 91 & 116 & 144 & 175 & 213 & 254 & 300 & 350 & 406 &  &  &  \\
\hline
30 & 27 & 38 & 50 & 66 & 83 & 101 & 120 & 168 & 216 & 264 & 312 & 365 & 435 &  &  \\
31 & 32 & 43 & 55 & 75 & 95 & 119 & 147 & 179 & 214 & 256 & 301 & 351 & 405 & 465 &  \\
32 & 23 & 33 & 46 & 60 & 76 & 96 & 116 & 148 & 184 & 224 & 264 & 312 & 364 & 420 & 496 \\
\end{tabular}
\caption{$W(4,\mathbb{Z}_n,D/n)$ for different $n$ and $D$.}
\label{tab:bubble4_2}
\end{center}
\end{table}

\newpage

\begin{table}
\begin{center}
\begin{tabular}{c|ccccc|ccccc|ccccc|ccccc}
n $\backslash$ D & 0 & 1 & 2 & 3 & 4 & 5 & 6 & 7 & 8 & 9 & 10 & 11 & 12 & 13 & 14 & 15 & 16 & 17   \\
\hline
5 & 0 & 0 & 0 & 0 & 0 & 10 &  &  &  &  &  &  &  &  &  &  &  &  \\
6 & 0 & 0 & 0 & 0 & 0 & 4 & 15 &  &  &  &  &  &  &  &  &  &  &  \\
7 & 0 & 0 & 0 & 0 & 0 & 1 & 6 & 21 &  &  &  &  &  &  &  &  &  &  \\
8 & 0 & 0 & 0 & 0 & 0 & 0 & 2 & 10 & 28 &  &  &  &  &  &  &  &  &  \\
9 & 0 & 0 & 0 & 0 & 0 & 0 & 2 & 8 & 18 & 36 &  &  &  &  &  &  &  &  \\
\hline
10 & 0 & 0 & 0 & 0 & 0 & 0 & 0 & 0 & 0 & 20 & 45 &  &  &  &  &  &  &  \\
11 & 0 & 0 & 0 & 0 & 0 & 0 & 0 & 2 & 6 & 15 & 30 & 55 &  &  &  &  &  &  \\
12 & 0 & 0 & 0 & 0 & 0 & 0 & 0 & 0 & 0 & 9 & 20 & 38 & 66 &  &  &  &  &  \\
13 & 0 & 0 & 0 & 0 & 0 & 0 & 0 & 0 & 1 & 6 & 14 & 28 & 48 & 78 &  &  &  &  \\
14 & 0 & 0 & 0 & 0 & 0 & 0 & 0 & 0 & 0 & 2 & 4 & 14 & 24 & 54 & 91 &  &  &  \\
\hline
15 & 0 & 0 & 0 & 0 & 0 & 0 & 0 & 0 & 0 & 2 & 6 & 9 & 12 & 42 & 72 & 105 &  &  \\
16 & 0 & 0 & 0 & 0 & 0 & 0 & 0 & 0 & 0 & 0 & 2 & 7 & 14 & 30 & 50 & 78 & 120 &  \\
17 & 0 & 0 & 0 & 0 & 0 & 0 & 0 & 0 & 0 & 0 & 1 & 6 & 14 & 26 & 42 & 66 & 96 & 136 \\
18 & 0 & 0 & 0 & 0 & 0 & 0 & 0 & 0 & 0 & 0 & 0 & 3 & 8 & 18 & 29 & 47 & 72 & 106 \\
19 & 0 & 0 & 0 & 0 & 0 & 0 & 0 & 0 & 0 & 0 & 0 & 1 & 3 & 12 & 24 & 41 & 62 & 91 \\
\hline
20 & 0 & 0 & 0 & 0 & 0 & 0 & 0 & 0 & 0 & 0 & 0 & 0 & 0 & 4 & 8 & 12 & 16 & 56 \\
21 & 0 & 0 & 0 & 0 & 0 & 0 & 0 & 0 & 0 & 0 & 0 & 0 & 0 & 5 & 13 & 24 & 38 & 54 \\
22 & 0 & 0 & 0 & 0 & 0 & 0 & 0 & 0 & 0 & 0 & 0 & 0 & 0 & 2 & 6 & 14 & 24 & 42 \\
23 & 0 & 0 & 0 & 0 & 0 & 0 & 0 & 0 & 0 & 0 & 0 & 0 & 0 & 0 & 2 & 10 & 21 & 38 \\
24 & 0 & 0 & 0 & 0 & 0 & 0 & 0 & 0 & 0 & 0 & 0 & 0 & 0 & 0 & 0 & 5 & 12 & 28 \\
\hline
25 & 0 & 0 & 0 & 0 & 0 & 0 & 0 & 0 & 0 & 0 & 0 & 0 & 0 & 0 & 0 & 0 & 0 & 10 \\
26 & 0 & 0 & 0 & 0 & 0 & 0 & 0 & 0 & 0 & 0 & 0 & 0 & 0 & 0 & 0 & 2 & 4 & 13 \\
27 & 0 & 0 & 0 & 0 & 0 & 0 & 0 & 0 & 0 & 0 & 0 & 0 & 0 & 0 & 0 & 0 & 4 & 13 \\
28 & 0 & 0 & 0 & 0 & 0 & 0 & 0 & 0 & 0 & 0 & 0 & 0 & 0 & 0 & 0 & 0 & 2 & 6 \\
29 & 0 & 0 & 0 & 0 & 0 & 0 & 0 & 0 & 0 & 0 & 0 & 0 & 0 & 0 & 0 & 0 & 0 & 4 \\
\hline
30 & 0 & 0 & 0 & 0 & 0 & 0 & 0 & 0 & 0 & 0 & 0 & 0 & 0 & 0 & 0 & 0 & 0 & 3 \\
31 & 0 & 0 & 0 & 0 & 0 & 0 & 0 & 0 & 0 & 0 & 0 & 0 & 0 & 0 & 0 & 0 & 0 & 2 \\
32 & 0 & 0 & 0 & 0 & 0 & 0 & 0 & 0 & 0 & 0 & 0 & 0 & 0 & 0 & 0 & 0 & 0 & 1 \\
\end{tabular}
\caption{$W(5,\mathbb{Z}_n,D/n)$ for different $n$ and $D$.}
\label{tab:bubble5_1}
\end{center}
\end{table}

\newpage

\begin{table}
\begin{center}
\begin{tabular}{c|ccccc|ccccc|ccccc|ccccc}
n $\backslash$ D & 18 & 19 & 20 & 21 & 22 & 23 & 24 & 25 & 26 & 27& 28 & 29 & 30 & 31 & 32    \\
\hline
5 &  &  &  &  &  &  &  &  &  &  &  &  &  &  & \\
6 &  &  &  &  &  &  &  &  &  &  &  &  &  &  & \\ 
7 &  &  &  &  &  &  &  &  &  &  &  &  &  &  &  \\
8 &  &  &  &  &  &  &  &  &  &  &  &  &  &  &  \\
9 &  &  &  &  &  &  &  &  &  &  &  &  &  &  &  \\
\hline
10 &  &  &  &  &  &  &  &  &  &  &  &  &  &  &  \\
11 &  &  &  &  &  &  &  &  &  &  &  &  &  &  &  \\
12 &  &  &  &  &  &  &  &  &  &  &  &  &  &  &  \\
13 &  &  &  &  &  &  &  &  &  &  &  &  &  &  &  \\
14 &  &  &  &  &  &  &  &  &  &  &  &  &  &  &  \\
\hline
15 &  &  &  &  &  &  &  &  &  &  &  &  &  &  &  \\
16 &  &  &  &  &  &  &  &  &  &  &  &  &  &  &  \\
17 &  &  &  &  &  &  &  &  &  &  &  &  &  &  &  \\
18 & 153 &  &  &  &  &  &  &  &  &  &  &  &  &  & \\ 
19 & 126 & 171 &  &  &  &  &  &  &  &  &  &  &  &  & \\
\hline 
20 & 96 & 136 & 190 &  &  &  &  &  &  &  &  &  &  &  & \\ 
21 & 72 & 117 & 162 & 210 &  &  &  &  &  &  &  &  &  &  & \\ 
22 & 60 & 90 & 120 & 170 & 231 &  &  &  &  &  &  &  &  &  & \\ 
23 & 56 & 83 & 114 & 153 & 198 & 253 &  &  &  &  &  &  &  &  & \\ 
24 & 44 & 65 & 88 & 123 & 164 & 212 & 276 &  &  &  &  &  &  &  & \\ 
\hline
25 & 20 & 30 & 40 & 90 & 140 & 190 & 240 & 300 &  &  &  &  &  &  & \\ 
26 & 24 & 40 & 56 & 84 & 112 & 152 & 192 & 252 & 325 &  &  &  &  &  & \\ 
27 & 24 & 39 & 58 & 81 & 108 & 143 & 182 & 233 & 288 & 351 &  &  &  &  & \\ 
28 & 10 & 23 & 36 & 55 & 74 & 96 & 120 & 180 & 240 & 300 & 378 &  &  &  &  \\
29 & 11 & 21 & 34 & 50 & 74 & 103 & 134 & 176 & 222 & 276 & 336 & 406 &  &  & \\ 
30 & 8 & 12 & 16 & 24 & 32 & 40 & 48 & 108 & 168 & 228 & 288 & 352 & 435 &  &  \\
\hline
31 & 5 & 10 & 20 & 33 & 49 & 72 & 100 & 133 & 168 & 215 & 266 & 325 & 390 & 465 & \\ 
32 & 2 & 8 & 17 & 26 & 38 & 56 & 76 & 106 & 139 & 179 & 222 & 278 & 338 & 406 & 496 \\
\end{tabular}
\caption{$W(5,\mathbb{Z}_n,D/n)$ for different $n$ and $D$.}
\label{tab:bubble5_2}
\end{center}
\end{table}

\newpage

\end{document}